\documentclass{article}%
\usepackage{amsfonts}
\usepackage{graphicx}
\usepackage{amsmath}%
\usepackage{amssymb}
\providecommand{\U}[1]{\protect\rule{.1in}{.1in}}
\newtheorem{theorem}{Theorem}

\newtheorem{proposition}[theorem]{Proposition}

\newenvironment{proof}[1][Proof]{\textbf{#1.} }{\ \rule{0.5em}{0.5em}}
\begin{document}

\title{A Note on the Infinitesimal Baker-Campbell-Hausdorff Formula}
\author{Hirokazu NISHIMURA$^{1}$ and Hirowaki TAKAMIYA$^{2}$\\1 Institute of Mathematics, University of Tsukuba\\\ \ Tsukuba, Ibaraki 305-8571\\\ \ Japan\\2 Naha Commercial High School, 16-1, Matsuyama-1\\\ \ Naha, Okinawa 900-0032\\\ \ Japan}
\maketitle

\begin{abstract}
We have studied the infinitesimal Baker-Campbell-Hausdorff formula up to $n=4
$ (Math. Appl. \textbf{2} (2013), 61-91). In this note we correct some errors
in our calculation for $n=4$\ and presents the calculation for $n=5$ by using Mathematica.

\end{abstract}

\section{Introduction}

We are going to work within synthetic differential geometry, in which a
\textit{Lie group} $G$\ is a group and a microlinear space at the same time.
For synthetic differential geometry, the reader is referred to \cite{ko} and
\cite{la}. Its \textit{Lie algebra} (i.e., its tangent space $\mathbf{T}_{e}G
$\ of $G$\ at the identity $e\in G$), usually denoted by $\mathfrak{g}$, is
endowed with a Lie bracket $\left[  ,\right]  $ abiding by antisymmetry and
the Jacobi identity. Each element $X\in\mathfrak{g}$ is a mapping $d\in
D\mapsto X_{d}\in G$ with $X_{0}=e$, where
\[
D=\left\{  d\in\mathbb{R}\mid d^{2}=0\right\}
\]

We assume that the so-called \textit{exponential mapping} $\exp:\mathfrak{g}%
\rightarrow G$ exists. The infinitesimal Baker-Campbell-Hausdorff formula
expresses
\[
\exp\;\left(  d_{1}+...+d_{n}\right)  X.\exp\;\left(  d_{1}+...+d_{n}\right)
Y
\]
as
\[
\exp\;\left(  \text{a Lie polynomial of }X\text{\ and }Y\right)
\]
where $X,Y\in\mathfrak{g}$ and $d_{1},...,d_{n}\in D$. In \cite{ni} we have
calculated the infinitesimal Baker-Campbell-Hausdorff formula up to $n=4$, but
the second author \cite{ta} found out some errors in the calculation for $n=4$.

This paper is based upon \cite{ta}. We correct some errors in the calculation
of the infinitesimal Baker-Campbell-Hausdorff formula in case of $n=4$ in our
previous paper \cite{ni} and we present a calculation of the infinitesimal
Baker-Campbell-Hausdorff formula in case of $n=5$ newly. Both calculations
were implemented by using Mathematica.

\section{Preliminaries}

The infinitesimal Baker-Campbell-Hausdorff formula for $n=3$ goes as follows:

\begin{theorem}
\label{t2.0}(cf. Theorem 7.5 and Theorem 8.3 of \cite{ni}) Given
$X,Y\in\mathfrak{g}$\ and $d_{1},d_{2},d_{3}\in D$, we have
\begin{align*}
& \exp\,\left(  d_{1}+d_{2}+d_{3}\right)  X.\exp\,\left(  d_{1}+d_{2}%
+d_{3}\right)  Y\\
& =\exp\,\left(  d_{1}+d_{2}+d_{3}\right)  \left(  X+Y\right)  +\frac{1}%
{2}\left(  d_{1}+d_{2}+d_{3}\right)  ^{2}\left[  X,Y\right]  +\\
& \frac{1}{12}\left(  d_{1}+d_{2}+d_{3}\right)  ^{3}\left[  X-Y,\left[
X,Y\right]  \right]
\end{align*}

\end{theorem}

The tangent space $\mathbf{T}_{X}\mathfrak{g}$ of $\mathfrak{g}$\ at
$X\in\mathfrak{g}$\ is naturally identified with $\mathfrak{g}$\ itself. That
is to say, each $Y\in\mathfrak{g}$ gives rise to $\left(  d\in D\mapsto
X+dY\in\mathfrak{g}\right)  \in\mathbf{T}_{X}\mathfrak{g}$, which yields a
bijection between $\mathfrak{g}$\ and $\mathbf{T}_{X}\mathfrak{g}$. Its
\textit{left logarithmic derivative} $\delta^{\mathrm{left}}\left(
\exp\right)  $ and its \textit{right logarithmic derivative} $\delta
^{\mathrm{right}}\left(  \exp\right)  $ are characterized by the following
formulas:
\begin{equation}
\exp\;X+dY=\exp\;X.\left(  \delta^{\mathrm{left}}\left(  \exp\right)  \left(
X\right)  \left(  Y\right)  \right)  _{d}\label{2.1}%
\end{equation}
and
\begin{equation}
\exp\;X+dY=\left(  \delta^{\mathrm{right}}\left(  \exp\right)  \left(
X\right)  \left(  Y\right)  \right)  _{d}.\exp\;X\label{2.2}%
\end{equation}
for any $X,Y\in\mathfrak{g}$ and any $d\in D$. For logarithmic derivatives,
the reader is referred to \S 5 of \cite{ni} and \S 38.1 of \cite{krmi}. We
have the following well-known formulas.

\begin{theorem}
\label{t2.1}(cf. Theorem 5.3 and Theorem 5.8 of \cite{ni}) Given
$X\in\mathfrak{g}$\ with $\left(  \mathrm{ad}\,X\right)  ^{n+1}$ vanishing for
some natural number $n$, we have
\[
\delta^{\mathrm{left}}\left(  \exp\right)  \left(  X\right)  =\sum_{p=0}%
^{n}\frac{\left(  -1\right)  ^{p}}{\left(  p+1\right)  !}\left(
\mathrm{ad}\,X\right)  ^{p}%
\]
and
\[
\delta^{\mathrm{right}}\left(  \exp\right)  \left(  X\right)  =\sum_{p=0}%
^{n}\frac{1}{\left(  p+1\right)  !}\left(  \mathrm{ad}\,X\right)  ^{p}%
\]

\end{theorem}

We note in passing that

\begin{proposition}
\label{t2.2}(cf. Proposition 5.4 of \cite{ni}) For any $X,Y\in\mathfrak{g}%
$\ with $\left[  X,Y\right]  $ vanishing, we have
\[
\exp\;X.\exp\;Y=\exp\;X+Y
\]

\end{proposition}

The following simple proposition is very useful.

\begin{proposition}
\label{t2.3}(cf. Proposition 4.9 of \cite{ni}) For any $X\in\mathfrak{g}$ and
any $d\in D$, we have
\[
\exp\;dX=X_{d}%
\]

\end{proposition}

\section{The BCH Formula for n=4}

\begin{theorem}
\label{t3.1}
\begin{align*}
& \exp\;\left(  d_{1}+d_{2}+d_{3}+d_{4}\right)  X.\exp\;\left(  d_{1}%
+d_{2}+d_{3}+d_{4}\right)  Y\\
& =\exp\;\left(  d_{1}+d_{2}+d_{3}+d_{4}\right)  \left(  X+Y\right) \\
& +\frac{1}{2}\left(  d_{1}+d_{2}+d_{3}+d_{4}\right)  ^{2}\left[  X,Y\right]
\\
& +\frac{1}{12}\left(  d_{1}+d_{2}+d_{3}+d_{4}\right)  ^{3}\left[  \left[
X,Y\right]  ,Y-X\right] \\
& +\frac{1}{96}\left(  d_{1}+d_{2}+d_{3}+d_{4}\right)  ^{4}\left(  \left[
X+Y,\left[  \left[  X,Y\right]  ,X+Y\right]  \right]  +\left[  \left[  \left[
X,Y\right]  ,X-Y\right]  ,X-Y\right]  \right)
\end{align*}

\end{theorem}

\begin{proof}
We have
\begin{align}
& \exp\;\left(  d_{1}+d_{2}+d_{3}+d_{4}\right)  X.\exp\;\left(  d_{1}%
+d_{2}+d_{3}+d_{4}\right)  Y\nonumber\\
& =\exp\;\left(  d_{1}+d_{2}+d_{3}\right)  X+d_{4}X.\exp\;\left(  d_{1}%
+d_{2}+d_{3}\right)  Y+d_{4}Y\nonumber\\
& =\exp\;d_{4}X.\exp\;\left(  d_{1}+d_{2}+d_{3}\right)  X.\exp\;\left(
d_{1}+d_{2}+d_{3}\right)  Y.\exp\;d_{4}Y\nonumber\\
& \left)  \text{By Proposition \ref{t2.2}}\right( \nonumber\\
& =\exp\;d_{4}X.\nonumber\\
& \exp\;\left(  \left(  d_{1}+d_{2}+d_{3}\right)  \left(  X+Y\right)
+\frac{1}{2}\left(  d_{1}+d_{2}+d_{3}\right)  ^{2}\left[  X,Y\right]
+\frac{1}{2}d_{1}d_{2}d_{3}\left[  \left[  X,Y\right]  ,Y-X\right]  \right)
.\nonumber\\
& \exp\;d_{4}Y\nonumber\\
& \left)  \text{By Theorem \ref{t2.0}}\right( \label{3.1f}%
\end{align}

By the way, due to Theorem \ref{t2.1}, we have
\begin{align}
& \delta^{\mathrm{left}}\left(  \exp\right)  \left(
\begin{array}
[c]{c}%
\left(  d_{1}+d_{2}+d_{3}\right)  \left(  X+Y\right)  +\frac{1}{2}\left(
d_{1}+d_{2}+d_{3}\right)  ^{2}\left[  X,Y\right] \\
+\frac{1}{2}d_{1}d_{2}d_{3}\left[  \left[  X,Y\right]  ,Y-X\right]
\end{array}
\right)  \left(  Y\right) \nonumber\\
& =-\frac{1}{2}d_{1}\left[  X,Y\right]  -\frac{1}{2}d_{2}\left[  X,Y\right]
+\frac{1}{3}d_{1}d_{2}\left[  X,\left[  X,Y\right]  \right]  +\frac{1}{3}%
d_{1}d_{2}\left[  Y,\left[  X,Y\right]  \right] \nonumber\\
& -\frac{1}{2}d_{1}d_{2}\left[  \left[  X,Y\right]  ,Y\right]  -\frac{1}%
{2}d_{3}\left[  X,Y\right]  +\frac{1}{3}d_{1}d_{3}\left[  X,\left[
X,Y\right]  \right]  +\frac{1}{3}d_{1}d_{3}\left[  Y,\left[  X,Y\right]
\right] \nonumber\\
& -\frac{1}{2}d_{1}d_{3}\left[  \left[  X,Y\right]  ,Y\right]  +\frac{1}%
{3}d_{2}d_{3}\left[  X,\left[  X,Y\right]  \right]  +\frac{1}{3}d_{2}%
d_{3}\left[  Y,\left[  X,Y\right]  \right] \nonumber\\
& -\frac{1}{2}d_{2}d_{3}\left[  \left[  X,Y\right]  ,Y\right]  -\frac{1}%
{4}d_{1}d_{2}d_{3}\left[  X,\left[  X,\left[  X,Y\right]  \right]  \right]
-\frac{1}{4}d_{1}d_{2}d_{3}\left[  X,\left[  Y,\left[  X,Y\right]  \right]
\right] \nonumber\\
& +\frac{1}{2}d_{1}d_{2}d_{3}\left[  X,\left[  \left[  X,Y\right]  ,Y\right]
\right]  -\frac{1}{4}d_{1}d_{2}d_{3}\left[  Y,\left[  X,\left[  X,Y\right]
\right]  \right] \nonumber\\
& -\frac{1}{4}d_{1}d_{2}d_{3}\left[  Y,\left[  Y,\left[  X,Y\right]  \right]
\right]  +\frac{1}{2}d_{1}d_{2}d_{3}\left[  Y,\left[  \left[  X,Y\right]
,Y\right]  \right] \nonumber\\
& +\frac{1}{4}d_{1}d_{2}d_{3}\left[  \left[  \left[  X,Y\right]  ,X\right]
,Y\right]  -\frac{1}{4}d_{1}d_{2}d_{3}\left[  \left[  \left[  X,Y\right]
,Y\right]  ,Y\right]  +Y\label{3.2}%
\end{align}

Letting $n_{41}$ be the right-hand side of (\ref{3.2}) with the last term $Y$
deleted, we have
\begin{align}
& \left(  \ref{3.1f}\right) \nonumber\\
& =\exp\;d_{4}X.\nonumber\\
& \exp\;\left(  \left(  d_{1}+d_{2}+d_{3}\right)  \left(  X+Y\right)
+\frac{1}{2}\left(  d_{1}+d_{2}+d_{3}\right)  ^{2}\left[  X,Y\right]
+\frac{1}{2}d_{1}d_{2}d_{3}\left[  \left[  X,Y\right]  ,Y-X\right]  \right)
.\nonumber\\
& \exp\;d_{4}\left(  n_{41}+Y\right)  .\exp\;-d_{4}n_{41}\nonumber\\
& \left)  \text{By Proposition \ref{t2.2}}\right( \nonumber\\
& =\exp\;d_{4}X.\nonumber\\
& \exp\;\left(  \left(  d_{1}+d_{2}+d_{3}\right)  \left(  X+Y\right)
+\frac{1}{2}\left(  d_{1}+d_{2}+d_{3}\right)  ^{2}\left[  X,Y\right]
+\frac{1}{2}d_{1}d_{2}d_{3}\left[  \left[  X,Y\right]  ,Y-X\right]  \right)
.\nonumber\\
& \left(  n_{41}+Y\right)  _{d_{4}}.\exp\;-d_{4}n_{41}\nonumber\\
& \left)  \text{By Proposition \ref{t2.3}}\right( \nonumber\\
& =\exp\;d_{4}X.\exp\;\left(
\begin{array}
[c]{c}%
\left(  d_{1}+d_{2}+d_{3}\right)  \left(  X+Y\right)  +\frac{1}{2}\left(
d_{1}+d_{2}+d_{3}\right)  ^{2}\left[  X,Y\right] \\
+\frac{1}{2}d_{1}d_{2}d_{3}\left[  \left[  X,Y\right]  ,Y-X\right]  +d_{4}Y
\end{array}
\right)  .\exp\;-d_{4}n_{41}\nonumber\\
& \left)  \text{By (\ref{3.2})}\right( \label{3.3}%
\end{align}

We let $i_{41}$ be the result of $n_{41}$ by deleting all the terms whose
coefficients contain $d_{1}d_{2}d_{3}$. Then, due to Theorem \ref{t2.1}, we
have
\begin{align}
& \delta^{\mathrm{left}}\left(  \exp\right)  \left(
\begin{array}
[c]{c}%
\left(  d_{1}+d_{2}+d_{3}\right)  \left(  X+Y\right)  +\frac{1}{2}\left(
d_{1}+d_{2}+d_{3}\right)  ^{2}\left[  X,Y\right] \\
+\frac{1}{2}d_{1}d_{2}d_{3}\left[  \left[  X,Y\right]  ,Y-X\right]  +d_{4}Y
\end{array}
\right)  \left(  i_{41}\right) \nonumber\\
& =\frac{1}{2}d_{1}\left[  X,Y\right]  +\frac{1}{2}d_{2}\left[  X,Y\right]
-\frac{5}{6}d_{1}d_{2}\left[  X,\left[  X,Y\right]  \right]  -\frac{5}{6}%
d_{1}d_{2}\left[  Y,\left[  X,Y\right]  \right] \nonumber\\
& +\frac{1}{2}d_{1}d_{2}\left[  \left[  X,Y\right]  ,Y\right]  +\frac{1}%
{2}d_{3}\left[  X,Y\right]  -\frac{5}{6}d_{1}d_{3}\left[  X,\left[
X,Y\right]  \right]  -\frac{5}{6}d_{1}d_{3}\left[  Y,\left[  X,Y\right]
\right] \nonumber\\
& +\frac{1}{2}d_{1}d_{3}\left[  \left[  X,Y\right]  ,Y\right]  -\frac{5}%
{6}d_{2}d_{3}\left[  X,\left[  X,Y\right]  \right]  -\frac{5}{6}d_{2}%
d_{3}\left[  Y,\left[  X,Y\right]  \right] \nonumber\\
& +\frac{1}{2}d_{2}d_{3}\left[  \left[  X,Y\right]  ,Y\right]  +d_{1}%
d_{2}d_{3}\left[  X,\left[  X,\left[  X,Y\right]  \right]  \right] \nonumber\\
& +d_{1}d_{2}d_{3}\left[  X,\left[  Y,\left[  X,Y\right]  \right]  \right]
-\frac{3}{4}d_{1}d_{2}d_{3}\left[  X,\left[  \left[  X,Y\right]  ,Y\right]
\right]  +d_{1}d_{2}d_{3}\left[  Y,\left[  X,\left[  X,Y\right]  \right]
\right] \nonumber\\
& +d_{1}d_{2}d_{3}\left[  Y,\left[  Y,\left[  X,Y\right]  \right]  \right]
-\frac{3}{4}d_{1}d_{2}d_{3}\left[  Y,\left[  \left[  X,Y\right]  ,Y\right]
\right] \label{3.4}%
\end{align}

Letting $n_{42}$ be the right-hand side of (\ref{3.4}), we have
\begin{align}
& \left(  \ref{3.3}\right) \nonumber\\
& =\exp\;d_{4}X.\exp\;\left(
\begin{array}
[c]{c}%
\left(  d_{1}+d_{2}+d_{3}\right)  \left(  X+Y\right)  +\frac{1}{2}\left(
d_{1}+d_{2}+d_{3}\right)  ^{2}\left[  X,Y\right] \\
+\frac{1}{2}d_{1}d_{2}d_{3}\left[  \left[  X,Y\right]  ,Y-X\right]  +d_{4}Y
\end{array}
\right)  .\nonumber\\
& \exp\;d_{4}n_{42}.\exp\;-d_{4}n_{42}-d_{4}n_{41}\nonumber\\
& \left)  \text{By Proposition \ref{t2.2}}\right( \nonumber\\
& =\exp\;d_{4}X.\exp\;\left(
\begin{array}
[c]{c}%
\left(  d_{1}+d_{2}+d_{3}\right)  \left(  X+Y\right)  +\frac{1}{2}\left(
d_{1}+d_{2}+d_{3}\right)  ^{2}\left[  X,Y\right] \\
+\frac{1}{2}d_{1}d_{2}d_{3}\left[  \left[  X,Y\right]  ,Y-X\right]  +d_{4}Y
\end{array}
\right)  .\nonumber\\
& \left(  n_{42}\right)  _{d_{4}}.\exp\;-d_{4}n_{42}-d_{4}n_{41}\nonumber\\
& \left)  \text{By Proposition \ref{t2.3}}\right( \nonumber\\
& =\exp\;d_{4}X.\exp\;\left(
\begin{array}
[c]{c}%
\left(  d_{1}+d_{2}+d_{3}\right)  \left(  X+Y\right)  +\frac{1}{2}\left(
d_{1}+d_{2}+d_{3}\right)  ^{2}\left[  X,Y\right] \\
+\frac{1}{2}d_{1}d_{2}d_{3}\left[  \left[  X,Y\right]  ,Y-X\right]
+d_{4}Y+d_{4}i_{41}%
\end{array}
\right)  .\nonumber\\
& \exp\;-d_{4}n_{42}-d_{4}n_{41}\nonumber\\
& \left)  \text{By (\ref{3.4})}\right( \label{3.5}%
\end{align}

We let $i_{42}$ be the result of $-n_{42}-n_{41}$ by deleting all the terms
whose coefficients contain $d_{1}d_{2}d_{3}$. Then, thanks to Theorem
\ref{t2.1}, we have
\begin{align}
& \delta^{\mathrm{left}}\left(  \exp\right)  \left(
\begin{array}
[c]{c}%
\left(  d_{1}+d_{2}+d_{3}\right)  \left(  X+Y\right)  +\frac{1}{2}\left(
d_{1}+d_{2}+d_{3}\right)  ^{2}\left[  X,Y\right] \\
+\frac{1}{2}d_{1}d_{2}d_{3}\left[  \left[  X,Y\right]  ,Y-X\right]
+d_{4}Y+d_{4}i_{41}%
\end{array}
\right)  \left(  i_{42}\right) \nonumber\\
& =\frac{1}{2}d_{1}d_{2}\left[  X,\left[  X,Y\right]  \right]  +\frac{1}%
{2}d_{1}d_{2}\left[  Y,\left[  X,Y\right]  \right]  +\frac{1}{2}d_{1}%
d_{3}\left[  X,\left[  X,Y\right]  \right] \nonumber\\
& +\frac{1}{2}d_{1}d_{3}\left[  Y,\left[  X,Y\right]  \right]  +\frac{1}%
{2}d_{2}d_{3}\left[  X,\left[  X,Y\right]  \right]  +\frac{1}{2}d_{2}%
d_{3}\left[  Y,\left[  X,Y\right]  \right] \nonumber\\
& -\frac{3}{4}d_{1}d_{2}d_{3}\left[  X,\left[  X,\left[  X,Y\right]  \right]
\right]  -\frac{3}{4}d_{1}d_{2}d_{3}\left[  X,\left[  Y,\left[  X,Y\right]
\right]  \right] \nonumber\\
& -\frac{3}{4}d_{1}d_{2}d_{3}\left[  Y,\left[  X,\left[  X,Y\right]  \right]
\right]  -\frac{3}{4}d_{1}d_{2}d_{3}\left[  Y,\left[  Y,\left[  X,Y\right]
\right]  \right] \label{3.6}%
\end{align}

Letting $n_{43}$ be the right-hand side of (\ref{3.6}), we have
\begin{align}
& \left(  \ref{3.5}\right) \nonumber\\
& =\exp\;d_{4}X.\exp\;\left(
\begin{array}
[c]{c}%
\left(  d_{1}+d_{2}+d_{3}\right)  \left(  X+Y\right)  +\frac{1}{2}\left(
d_{1}+d_{2}+d_{3}\right)  ^{2}\left[  X,Y\right] \\
+\frac{1}{2}d_{1}d_{2}d_{3}\left[  \left[  X,Y\right]  ,Y-X\right]
+d_{4}Y+d_{4}i_{41}%
\end{array}
\right)  .\nonumber\\
& \exp\;d_{4}n_{43}.\exp\;-d_{4}n_{43}-d_{4}n_{42}-d_{4}n_{41}\nonumber\\
& \left)  \text{By Proposition \ref{t2.2}}\right( \nonumber\\
& =\exp\;d_{4}X.\exp\;\left(
\begin{array}
[c]{c}%
\left(  d_{1}+d_{2}+d_{3}\right)  \left(  X+Y\right)  +\frac{1}{2}\left(
d_{1}+d_{2}+d_{3}\right)  ^{2}\left[  X,Y\right] \\
+\frac{1}{2}d_{1}d_{2}d_{3}\left[  \left[  X,Y\right]  ,Y-X\right]
+d_{4}Y+d_{4}i_{41}%
\end{array}
\right)  .\nonumber\\
& \left(  n_{43}\right)  _{d_{4}}.\exp\;-d_{4}n_{43}-d_{4}n_{42}-d_{4}%
n_{41}\nonumber\\
& \left)  \text{By Proposition \ref{t2.3}}\right( \nonumber\\
& =\exp\;d_{4}X.\exp\;\left(
\begin{array}
[c]{c}%
\left(  d_{1}+d_{2}+d_{3}\right)  \left(  X+Y\right)  +\frac{1}{2}\left(
d_{1}+d_{2}+d_{3}\right)  ^{2}\left[  X,Y\right] \\
+\frac{1}{2}d_{1}d_{2}d_{3}\left[  \left[  X,Y\right]  ,Y-X\right]
+d_{4}Y+d_{4}i_{41}+d_{4}i_{42}%
\end{array}
\right)  .\nonumber\\
& \exp\;-d_{4}n_{43}-d_{4}n_{42}-d_{4}n_{41}\nonumber\\
& \left)  \text{By (\ref{3.6})}\right( \label{3.7}%
\end{align}

Since the coefficient of every term in $-n_{43}-n_{42}-n_{41}$ contains
$d_{1}d_{2}d_{3}$, we now turn our attention to the left $\exp\;d_{4}X$. Now,
thanks to Theorem \ref{t2.1}, we have
\begin{align}
& \delta^{\mathrm{right}}\left(  \exp\right)  \left(
\begin{array}
[c]{c}%
\left(  d_{1}+d_{2}+d_{3}\right)  \left(  X+Y\right)  +\frac{1}{2}\left(
d_{1}+d_{2}+d_{3}\right)  ^{2}\left[  X,Y\right] \\
+\frac{1}{2}d_{1}d_{2}d_{3}\left[  \left[  X,Y\right]  ,Y-X\right]
+d_{4}Y+d_{4}i_{41}+d_{4}i_{42}%
\end{array}
\right)  \left(  X\right) \nonumber\\
& =-\frac{1}{2}d_{1}\left[  X,Y\right]  -\frac{1}{2}d_{2}\left[  X,Y\right]
-\frac{1}{3}d_{1}d_{2}\left[  X,\left[  X,Y\right]  \right]  -\frac{1}{3}%
d_{1}d_{2}\left[  Y,\left[  X,Y\right]  \right] \nonumber\\
& +\frac{1}{2}d_{1}d_{2}\left[  \left[  X,Y\right]  ,X\right]  -\frac{1}%
{2}d_{3}\left[  X,Y\right]  -\frac{1}{3}d_{1}d_{3}\left[  X,\left[
X,Y\right]  \right]  -\frac{1}{3}d_{1}d_{3}\left[  Y,\left[  X,Y\right]
\right] \nonumber\\
& +\frac{1}{2}d_{1}d_{3}\left[  \left[  X,Y\right]  ,X\right]  -\frac{1}%
{3}d_{2}d_{3}\left[  X,\left[  X,Y\right]  \right]  -\frac{1}{3}d_{2}%
d_{3}\left[  Y,\left[  X,Y\right]  \right] \nonumber\\
& +\frac{1}{2}d_{2}d_{3}\left[  \left[  X,Y\right]  ,X\right]  -\frac{1}%
{4}d_{1}d_{2}d_{3}\left[  X,\left[  X,\left[  X,Y\right]  \right]  \right]
-\frac{1}{4}d_{1}d_{2}d_{3}\left[  X,\left[  Y,\left[  X,Y\right]  \right]
\right] \nonumber\\
& +\frac{1}{2}d_{1}d_{2}d_{3}\left[  X,\left[  \left[  X,Y\right]  ,X\right]
\right]  -\frac{1}{4}d_{1}d_{2}d_{3}\left[  Y,\left[  X,\left[  X,Y\right]
\right]  \right] \nonumber\\
& -\frac{1}{4}d_{1}d_{2}d_{3}\left[  Y,\left[  Y,\left[  X,Y\right]  \right]
\right]  +\frac{1}{2}d_{1}d_{2}d_{3}\left[  Y,\left[  \left[  X,Y\right]
,X\right]  \right] \nonumber\\
& -\frac{1}{2}d_{1}d_{2}d_{3}\left[  \left[  X,Y\right]  ,\left[  X,Y\right]
\right]  -\frac{1}{4}d_{1}d_{2}d_{3}\left[  \left[  \left[  X,Y\right]
,X\right]  ,X\right] \nonumber\\
& +\frac{1}{4}d_{1}d_{2}d_{3}\left[  \left[  \left[  X,Y\right]  ,Y\right]
,X\right]  +X\label{3.8}%
\end{align}

We let $m_{41}$ be the right-hand side of (\ref{3.8}) with the last term
$X$\ deleted. Then we have
\begin{align}
& \left(  \ref{3.7}\right) \nonumber\\
& =\exp\;-d_{4}m_{41}.\exp\;d_{4}X+d_{4}m_{41}.\nonumber\\
& \exp\;\left(
\begin{array}
[c]{c}%
\left(  d_{1}+d_{2}+d_{3}\right)  \left(  X+Y\right)  +\frac{1}{2}\left(
d_{1}+d_{2}+d_{3}\right)  ^{2}\left[  X,Y\right] \\
+\frac{1}{2}d_{1}d_{2}d_{3}\left[  \left[  X,Y\right]  ,Y-X\right]
+d_{4}Y+d_{4}i_{41}+d_{4}i_{42}%
\end{array}
\right)  .\nonumber\\
& \exp\;-d_{4}n_{43}-d_{4}n_{42}-d_{4}n_{41}\nonumber\\
& \left)  \text{By Proposition \ref{t2.2}}\right( \nonumber\\
& =\exp\;-d_{4}m_{41}.\left(  X+m_{41}\right)  _{d_{4}}.\nonumber\\
& \exp\;\left(
\begin{array}
[c]{c}%
\left(  d_{1}+d_{2}+d_{3}\right)  \left(  X+Y\right)  +\frac{1}{2}\left(
d_{1}+d_{2}+d_{3}\right)  ^{2}\left[  X,Y\right] \\
+\frac{1}{2}d_{1}d_{2}d_{3}\left[  \left[  X,Y\right]  ,Y-X\right]
+d_{4}Y+d_{4}i_{41}+d_{4}i_{42}%
\end{array}
\right)  .\nonumber\\
& \exp\;-d_{4}n_{43}-d_{4}n_{42}-d_{4}n_{41}\nonumber\\
& \left)  \text{By Proposition \ref{t2.3}}\right( \nonumber\\
& =\exp\;-d_{4}m_{41}.\nonumber\\
& \exp\;\left(
\begin{array}
[c]{c}%
\left(  d_{1}+d_{2}+d_{3}+d_{4}\right)  \left(  X+Y\right)  +\frac{1}%
{2}\left(  d_{1}+d_{2}+d_{3}\right)  ^{2}\left[  X,Y\right] \\
+\frac{1}{2}d_{1}d_{2}d_{3}\left[  \left[  X,Y\right]  ,Y-X\right]
+d_{4}i_{41}+d_{4}i_{42}%
\end{array}
\right)  .\nonumber\\
& \exp\;-d_{4}n_{43}-d_{4}n_{42}-d_{4}n_{41}\nonumber\\
& \left)  \text{By (\ref{3.8})}\right( \label{3.9}%
\end{align}

We let $j_{41}$ be the result of $-m_{41}$ by deleting all the terms whose
coefficients contain $d_{1}d_{2}d_{3}$. Then, thanks to Theorem \ref{t2.1}, we
have
\begin{align}
& \delta^{\mathrm{right}}\left(  \exp\right)  \left(
\begin{array}
[c]{c}%
\left(  d_{1}+d_{2}+d_{3}+d_{4}\right)  \left(  X+Y\right)  +\frac{1}%
{2}\left(  d_{1}+d_{2}+d_{3}\right)  ^{2}\left[  X,Y\right] \\
+\frac{1}{2}d_{1}d_{2}d_{3}\left[  \left[  X,Y\right]  ,Y-X\right]
+d_{4}i_{41}+d_{4}i_{42}%
\end{array}
\right)  \left(  j_{41}\right) \nonumber\\
& =\frac{1}{2}d_{1}\left[  X,Y\right]  +\frac{1}{2}d_{2}\left[  X,Y\right]
+\frac{5}{6}d_{1}d_{2}\left[  X,\left[  X,Y\right]  \right]  +\frac{5}{6}%
d_{1}d_{2}\left[  Y,\left[  X,Y\right]  \right] \nonumber\\
& -\frac{1}{2}d_{1}d_{2}\left[  \left[  X,Y\right]  ,X\right]  +\frac{1}%
{2}d_{3}\left[  X,Y\right]  +\frac{5}{6}d_{1}d_{3}\left[  X,\left[
X,Y\right]  \right] \nonumber\\
& +\frac{5}{6}d_{1}d_{3}\left[  Y,\left[  X,Y\right]  \right]  -\frac{1}%
{2}d_{1}d_{3}\left[  \left[  X,Y\right]  ,X\right]  +\frac{5}{6}d_{2}%
d_{3}\left[  X,\left[  X,Y\right]  \right] \nonumber\\
& +\frac{5}{6}d_{2}d_{3}\left[  Y,\left[  X,Y\right]  \right]  -\frac{1}%
{2}d_{2}d_{3}\left[  \left[  X,Y\right]  ,X\right]  +d_{1}d_{2}d_{3}\left[
X,\left[  X,\left[  X,Y\right]  \right]  \right] \nonumber\\
& +d_{1}d_{2}d_{3}\left[  X,\left[  Y,\left[  X,Y\right]  \right]  \right]
-\frac{3}{4}d_{1}d_{2}d_{3}\left[  X,\left[  \left[  X,Y\right]  ,X\right]
\right] \nonumber\\
& +d_{1}d_{2}d_{3}\left[  Y,\left[  X,\left[  X,Y\right]  \right]  \right]
+d_{1}d_{2}d_{3}\left[  Y,\left[  Y,\left[  X,Y\right]  \right]  \right]
\nonumber\\
& -\frac{3}{4}d_{1}d_{2}d_{3}\left[  Y,\left[  \left[  X,Y\right]  ,X\right]
\right]  +\frac{3}{4}d_{1}d_{2}d_{3}\left[  \left[  X,Y\right]  ,\left[
X,Y\right]  \right] \label{3.10}%
\end{align}

Letting $m_{42}$ be the right-hand side of (\ref{3.10}), we have
\begin{align}
& \left(  \ref{3.9}\right) \nonumber\\
& =\exp\;-d_{4}m_{41}-d_{4}m_{42}.\exp\;d_{4}m_{42}.\nonumber\\
& \exp\;\left(
\begin{array}
[c]{c}%
\left(  d_{1}+d_{2}+d_{3}+d_{4}\right)  \left(  X+Y\right)  +\frac{1}%
{2}\left(  d_{1}+d_{2}+d_{3}\right)  ^{2}\left[  X,Y\right] \\
+\frac{1}{2}d_{1}d_{2}d_{3}\left[  \left[  X,Y\right]  ,Y-X\right]
+d_{4}i_{41}+d_{4}i_{42}%
\end{array}
\right)  .\nonumber\\
& \exp\;-d_{4}n_{43}-d_{4}n_{42}-d_{4}n_{41}\nonumber\\
& \left)  \text{By Proposition \ref{t2.2}}\right( \nonumber\\
& =\exp\;-d_{4}m_{41}-d_{4}m_{42}.\left(  m_{42}\right)  _{d_{4}}.\nonumber\\
& \exp\;\left(
\begin{array}
[c]{c}%
\left(  d_{1}+d_{2}+d_{3}+d_{4}\right)  \left(  X+Y\right)  +\frac{1}%
{2}\left(  d_{1}+d_{2}+d_{3}\right)  ^{2}\left[  X,Y\right] \\
+\frac{1}{2}d_{1}d_{2}d_{3}\left[  \left[  X,Y\right]  ,Y-X\right]
+d_{4}i_{41}+d_{4}i_{42}%
\end{array}
\right)  .\nonumber\\
& \exp\;-d_{4}n_{43}-d_{4}n_{42}-d_{4}n_{41}\nonumber\\
& \left)  \text{By Proposition \ref{t2.3}}\right( \nonumber\\
& =\exp\;-d_{4}m_{41}-d_{4}m_{42}.\nonumber\\
& \exp\;\left(
\begin{array}
[c]{c}%
\left(  d_{1}+d_{2}+d_{3}+d_{4}\right)  \left(  X+Y\right)  +\frac{1}%
{2}\left(  d_{1}+d_{2}+d_{3}\right)  ^{2}\left[  X,Y\right] \\
+\frac{1}{2}d_{1}d_{2}d_{3}\left[  \left[  X,Y\right]  ,Y-X\right]
+d_{4}i_{41}+d_{4}j_{41}+d_{4}i_{42}%
\end{array}
\right)  .\nonumber\\
& \exp\;-d_{4}n_{43}-d_{4}n_{42}-d_{4}n_{41}\nonumber\\
& \left)  \text{By (\ref{3.10})}\right( \label{3.11}%
\end{align}

We let $j_{42}$ be the result of $-m_{41}-m_{42}$ by deleting all the terms
whose coefficients contain $d_{1}d_{2}d_{3}$. Then, thanks to Theorem
\ref{t2.1}, we have
\begin{align}
& \delta^{\mathrm{right}}\left(  \exp\right)  \left(
\begin{array}
[c]{c}%
\left(  d_{1}+d_{2}+d_{3}+d_{4}\right)  \left(  X+Y\right)  +\frac{1}%
{2}\left(  d_{1}+d_{2}+d_{3}\right)  ^{2}\left[  X,Y\right] \\
+\frac{1}{2}d_{1}d_{2}d_{3}\left[  \left[  X,Y\right]  ,Y-X\right]
+d_{4}i_{41}+d_{4}j_{41}+d_{4}i_{42}%
\end{array}
\right)  \left(  j_{42}\right) \nonumber\\
& =-\frac{1}{2}d_{1}d_{2}\left[  X,\left[  X,Y\right]  \right]  -\frac{1}%
{2}d_{1}d_{2}\left[  Y,\left[  X,Y\right]  \right]  -\frac{1}{2}d_{1}%
d_{3}\left[  X,\left[  X,Y\right]  \right] \nonumber\\
& -\frac{1}{2}d_{1}d_{3}\left[  Y,\left[  X,Y\right]  \right]  -\frac{1}%
{2}d_{2}d_{3}\left[  X,\left[  X,Y\right]  \right]  -\frac{1}{2}d_{2}%
d_{3}\left[  Y,\left[  X,Y\right]  \right] \nonumber\\
& -\frac{3}{4}d_{1}d_{2}d_{3}\left[  X,\left[  X,\left[  X,Y\right]  \right]
\right]  -\frac{3}{4}d_{1}d_{2}d_{3}\left[  X,\left[  Y,\left[  X,Y\right]
\right]  \right] \nonumber\\
& -\frac{3}{4}d_{1}d_{2}d_{3}\left[  Y,\left[  X,\left[  X,Y\right]  \right]
\right]  -\frac{3}{4}d_{1}d_{2}d_{3}\left[  Y,\left[  Y,\left[  X,Y\right]
\right]  \right] \label{3.12}%
\end{align}

Letting $m_{43}$ be the right-hand side of (\ref{3.12}), we have
\begin{align}
& \left(  \ref{3.11}\right) \nonumber\\
& =\exp\;-d_{4}m_{41}-d_{4}m_{42}-d_{4}m_{43}.\exp\;d_{4}m_{43}.\nonumber\\
& \exp\;\left(
\begin{array}
[c]{c}%
\left(  d_{1}+d_{2}+d_{3}+d_{4}\right)  \left(  X+Y\right)  +\frac{1}%
{2}\left(  d_{1}+d_{2}+d_{3}\right)  ^{2}\left[  X,Y\right] \\
+\frac{1}{2}d_{1}d_{2}d_{3}\left[  \left[  X,Y\right]  ,Y-X\right]
+d_{4}i_{41}+d_{4}j_{41}+d_{4}i_{42}%
\end{array}
\right)  .\nonumber\\
& \exp\;-d_{4}n_{43}-d_{4}n_{42}-d_{4}n_{41}\nonumber\\
& \left)  \text{By Proposition \ref{t2.2}}\right( \nonumber\\
& =\exp\;-d_{4}m_{41}-d_{4}m_{42}-d_{4}m_{43}.\left(  m_{43}\right)  _{d_{4}%
}\nonumber\\
& \exp\;\left(
\begin{array}
[c]{c}%
\left(  d_{1}+d_{2}+d_{3}+d_{4}\right)  \left(  X+Y\right)  +\frac{1}%
{2}\left(  d_{1}+d_{2}+d_{3}\right)  ^{2}\left[  X,Y\right] \\
+\frac{1}{2}d_{1}d_{2}d_{3}\left[  \left[  X,Y\right]  ,Y-X\right]
+d_{4}i_{41}+d_{4}j_{41}+d_{4}i_{42}%
\end{array}
\right)  .\nonumber\\
& \exp\;-d_{4}n_{43}-d_{4}n_{42}-d_{4}n_{41}\nonumber\\
& \left)  \text{By Proposition \ref{t2.3}}\right( \nonumber\\
& =\exp\;-d_{4}m_{41}-d_{4}m_{42}-d_{4}m_{43}.\nonumber\\
& \exp\;\left(
\begin{array}
[c]{c}%
\left(  d_{1}+d_{2}+d_{3}+d_{4}\right)  \left(  X+Y\right)  +\frac{1}%
{2}\left(  d_{1}+d_{2}+d_{3}\right)  ^{2}\left[  X,Y\right] \\
+\frac{1}{2}d_{1}d_{2}d_{3}\left[  \left[  X,Y\right]  ,Y-X\right]
+d_{4}i_{41}+d_{4}j_{41}+d_{4}i_{42}+d_{4}j_{42}%
\end{array}
\right)  .\nonumber\\
& \exp\;-d_{4}n_{43}-d_{4}n_{42}-d_{4}n_{41}\nonumber\\
& \left)  \text{By (\ref{3.12})}\right( \label{3.13}%
\end{align}

Since the coefficient of every term in $-m_{41}-m_{42}-m_{43}$ contains
$d_{1}d_{2}d_{3}$, we are done, so that we have
\begin{align}
& \left(  \ref{3.13}\right) \nonumber\\
& =\exp\;\left(
\begin{array}
[c]{c}%
\left(  d_{1}+d_{2}+d_{3}+d_{4}\right)  \left(  X+Y\right)  +\frac{1}%
{2}\left(  d_{1}+d_{2}+d_{3}\right)  ^{2}\left[  X,Y\right] \\
+\frac{1}{2}d_{1}d_{2}d_{3}\left[  \left[  X,Y\right]  ,Y-X\right]
+d_{4}i_{41}+d_{4}j_{41}+d_{4}i_{42}+d_{4}j_{42}\\
-d_{4}m_{41}-d_{4}m_{42}-d_{4}m_{43}-d_{4}n_{43}-d_{4}n_{42}-d_{4}n_{41}%
\end{array}
\right) \label{3.14}%
\end{align}
It is easy to see that
\begin{align*}
& j_{41}+i_{42}+j_{42}\\
& =d_{1}\left[  X,Y\right]  +d_{2}\left[  X,Y\right]  -\frac{1}{2}d_{1}%
d_{2}\left[  \left[  X,Y\right]  ,X\right]  +\frac{1}{2}d_{1}d_{2}\left[
\left[  X,Y\right]  ,Y\right] \\
& d_{3}\left[  X,Y\right]  -\frac{1}{2}d_{1}d_{3}\left[  \left[  X,Y\right]
,X\right]  +\frac{1}{2}d_{1}d_{3}\left[  \left[  X,Y\right]  ,Y\right] \\
& -\frac{1}{2}d_{2}d_{3}\left[  \left[  X,Y\right]  ,X\right]  +\frac{1}%
{2}d_{2}d_{3}\left[  \left[  X,Y\right]  ,Y\right]
\end{align*}
whereas
\begin{align*}
& -m_{41}-m_{42}-m_{43}-n_{43}-n_{42}-n_{41}\\
& =\frac{1}{4}d_{1}d_{2}d_{3}\left[  X,\left[  \left[  X,Y\right]  ,X\right]
\right]  +\frac{1}{4}d_{1}d_{2}d_{3}\left[  X,\left[  \left[  X,Y\right]
,Y\right]  \right] \\
& +\frac{1}{4}d_{1}d_{2}d_{3}\left[  Y,\left[  \left[  X,Y\right]  ,X\right]
\right]  +\frac{1}{4}d_{1}d_{2}d_{3}\left[  Y,\left[  \left[  X,Y\right]
,Y\right]  \right] \\
& +\frac{1}{4}d_{1}d_{2}d_{3}\left[  \left[  \left[  X,Y\right]  ,X\right]
,X\right]  -\frac{1}{4}d_{1}d_{2}d_{3}\left[  \left[  \left[  X,Y\right]
,X\right]  ,Y\right] \\
& -\frac{1}{4}d_{1}d_{2}d_{3}\left[  \left[  \left[  X,Y\right]  ,Y\right]
,X\right]  +\frac{1}{4}d_{1}d_{2}d_{3}\left[  \left[  \left[  X,Y\right]
,Y\right]  ,Y\right]
\end{align*}
Therefore we have the desired result.
\end{proof}

\section{The BCH Formula for n=5}

\begin{theorem}
\label{t4.1}
\begin{align*}
& \exp\;\left(  d_{1}+d_{2}+d_{3}+d_{4}+d_{5}\right)  X.\exp\;\left(
d_{1}+d_{2}+d_{3}+d_{4}+d_{5}\right)  Y\\
& =\exp\;\left(  d_{1}+d_{2}+d_{3}+d_{4}+d_{5}\right)  \left(  X+Y\right) \\
& +\frac{1}{2}\left(  d_{1}+d_{2}+d_{3}+d_{4}+d_{5}\right)  ^{2}\left[
X,Y\right] \\
& +\frac{1}{12}\left(  d_{1}+d_{2}+d_{3}+d_{4}+d_{5}\right)  ^{3}\left[
\left[  X,Y\right]  ,Y-X\right] \\
& +\frac{1}{96}\left(  d_{1}+d_{2}+d_{3}+d_{4}+d_{5}\right)  ^{4}\left(
\left[  X+Y,\left[  \left[  X,Y\right]  ,X+Y\right]  \right]  +\left[  \left[
\left[  X,Y\right]  ,X-Y\right]  ,X-Y\right]  \right) \\
& +\frac{1}{120}\left(  d_{1}+d_{2}+d_{3}+d_{4}+d_{5}\right)  ^{5}\left(
\begin{array}
[c]{c}%
\frac{5}{6}\left[  X+Y,\left[  X+Y,\left[  \left[  X,Y\right]  ,X-Y\right]
\right]  \right] \\
+\frac{1}{2}\left[  \left[  X,Y\right]  ,\left[  \left[  X,Y\right]
,X+Y\right]  \right] \\
+\frac{1}{8}\left[  \left[  X+Y,\left[  \left[  X,Y\right]  ,X+Y\right]
\right]  ,Y-X\right] \\
+\frac{1}{8}\left[  \left[  \left[  \left[  X,Y\right]  ,Y-X\right]
,X-Y\right]  ,X-Y\right]
\end{array}
\right)
\end{align*}

\end{theorem}

\begin{proof}
We have
\begin{align}
& \exp\;\left(  d_{1}+d_{2}+d_{3}+d_{4}+d_{5}\right)  X.\exp\;\left(
d_{1}+d_{2}+d_{3}+d_{4}+d_{5}\right)  Y\nonumber\\
& =\exp\;\left(  d_{1}+d_{2}+d_{3}+d_{4}\right)  X+d_{5}X.\exp\;\left(
d_{1}+d_{2}+d_{3}+d_{4}\right)  Y+d_{5}Y\nonumber\\
& =\exp\;d_{5}X.\exp\;\left(  d_{1}+d_{2}+d_{3}+d_{4}\right)  X.\exp\;\left(
d_{1}+d_{2}+d_{3}+d_{4}\right)  Y.\exp\;d_{5}Y\nonumber\\
& \left)  \text{By Proposition \ref{t2.2}}\right(  \nonumber\\
& =\exp\;d_{5}X.\nonumber\\
& \exp\;\left(
\begin{array}
[c]{c}%
\left(  d_{1}+d_{2}+d_{3}+d_{4}\right)  \left(  X+Y\right)  +\frac{1}%
{2}\left(  d_{1}+d_{2}+d_{3}+d_{4}\right)  ^{2}\left[  X,Y\right]  \\
+\frac{1}{12}\left(  d_{1}+d_{2}+d_{3}+d_{4}\right)  ^{3}\left[  \left[
X,Y\right]  ,Y-X\right]  \\
+\frac{1}{96}\left(  d_{1}+d_{2}+d_{3}+d_{4}\right)  ^{4}\left(  \left[
X+Y,\left[  \left[  X,Y\right]  ,X+Y\right]  \right]  +\left[  \left[  \left[
X,Y\right]  ,X-Y\right]  ,X-Y\right]  \right)
\end{array}
\right)  .\nonumber\\
& \exp\;d_{5}Y\nonumber\\
& \left)  \text{By Theorem \ref{t3.1}}\right(  \label{4.1}%
\end{align}

By the way, due to Theorem \ref{t2.1}, we have
\begin{align*}
& \delta^{\mathrm{left}}\left(  \exp\right)  \left(
\begin{array}
[c]{c}%
\left(  d_{1}+d_{2}+d_{3}+d_{4}\right)  \left(  X+Y\right)  +\frac{1}%
{2}\left(  d_{1}+d_{2}+d_{3}+d_{4}\right)  ^{2}\left[  X,Y\right] \\
+\frac{1}{12}\left(  d_{1}+d_{2}+d_{3}+d_{4}\right)  ^{3}\left[  \left[
X,Y\right]  ,Y-X\right] \\
+\frac{1}{96}\left(  d_{1}+d_{2}+d_{3}+d_{4}\right)  ^{4}\left[  X+Y,\left[
\left[  X,Y\right]  ,X+Y\right]  \right] \\
+\frac{1}{96}\left(  d_{1}+d_{2}+d_{3}+d_{4}\right)  ^{4}\left[  \left[
\left[  X,Y\right]  ,X-Y\right]  ,X-Y\right]
\end{array}
\right)  \left(  Y\right) \\
& =Y-\frac{1}{2}d_{1}\left[  X,Y\right]  -\frac{1}{2}d_{2}\left[  X,Y\right]
+\frac{1}{3}d_{1}d_{2}\left[  X,\left[  X,Y\right]  \right]  +\frac{1}{3}%
d_{1}d_{2}\left[  Y,\left[  X,Y\right]  \right] \\
& -\frac{1}{2}d_{1}d_{2}\left[  \left[  X,Y\right]  ,Y\right]  -\frac{1}%
{2}d_{3}\left[  X,Y\right]  +\frac{1}{3}d_{1}d_{3}\left[  X,\left[
X,Y\right]  \right]  +\frac{1}{3}d_{1}d_{3}\left[  Y,\left[  X,Y\right]
\right] \\
& -\frac{1}{2}d_{1}d_{3}\left[  \left[  X,Y\right]  ,Y\right]  +\frac{1}%
{3}d_{2}d_{3}\left[  X,\left[  X,Y\right]  \right]  +\frac{1}{3}d_{2}%
d_{3}\left[  Y,\left[  X,Y\right]  \right]  -\frac{1}{2}d_{2}d_{3}\left[
\left[  X,Y\right]  ,Y\right] \\
& -\frac{1}{4}d_{1}d_{2}d_{3}\left[  X,\left[  X,\left[  X,Y\right]  \right]
\right]  -\frac{1}{4}d_{1}d_{2}d_{3}\left[  X,\left[  Y,\left[  X,Y\right]
\right]  \right]  +\frac{1}{2}d_{1}d_{2}d_{3}\left[  X,\left[  \left[
X,Y\right]  ,Y\right]  \right] \\
& -\frac{1}{4}d_{1}d_{2}d_{3}\left[  Y,\left[  X,\left[  X,Y\right]  \right]
\right]  -\frac{1}{4}d_{1}d_{2}d_{3}\left[  Y,\left[  Y,\left[  X,Y\right]
\right]  \right]  +\frac{1}{2}d_{1}d_{2}d_{3}\left[  Y,\left[  \left[
X,Y\right]  ,Y\right]  \right] \\
& +\frac{1}{4}d_{1}d_{2}d_{3}\left[  \left[  \left[  X,Y\right]  ,X\right]
,Y\right]  -\frac{1}{4}d_{1}d_{2}d_{3}\left[  \left[  \left[  X,Y\right]
,Y\right]  ,Y\right]  -\frac{1}{2}d_{4}\left[  X,Y\right]  +\frac{1}{3}%
d_{1}d_{4}\left[  X,\left[  X,Y\right]  \right] \\
& +\frac{1}{3}d_{1}d_{4}\left[  Y,\left[  X,Y\right]  \right]  -\frac{1}%
{2}d_{1}d_{4}\left[  \left[  X,Y\right]  ,Y\right]  +\frac{1}{3}d_{2}%
d_{4}\left[  X,\left[  X,Y\right]  \right]  +\frac{1}{3}d_{2}d_{4}\left[
Y,\left[  X,Y\right]  \right] \\
& -\frac{1}{2}d_{2}d_{4}\left[  \left[  X,Y\right]  ,Y\right]  -\frac{1}%
{4}d_{1}d_{2}d_{4}\left[  X,\left[  X,\left[  X,Y\right]  \right]  \right]
-\frac{1}{4}d_{1}d_{2}d_{4}\left[  X,\left[  Y,\left[  X,Y\right]  \right]
\right] \\
& +\frac{1}{2}d_{1}d_{2}d_{4}\left[  X,\left[  \left[  X,Y\right]  ,Y\right]
\right]  -\frac{1}{4}d_{1}d_{2}d_{4}\left[  Y,\left[  X,\left[  X,Y\right]
\right]  \right]  -\frac{1}{4}d_{1}d_{2}d_{4}\left[  Y,\left[  Y,\left[
X,Y\right]  \right]  \right] \\
& +\frac{1}{2}d_{1}d_{2}d_{4}\left[  Y,\left[  \left[  X,Y\right]  ,Y\right]
\right]  +\frac{1}{4}d_{1}d_{2}d_{4}\left[  \left[  \left[  X,Y\right]
,X\right]  ,Y\right]  -\frac{1}{4}d_{1}d_{2}d_{4}\left[  \left[  \left[
X,Y\right]  ,Y\right]  ,Y\right] \\
& +\frac{1}{3}d_{3}d_{4}\left[  X,\left[  X,Y\right]  \right]  +\frac{1}%
{3}d_{3}d_{4}\left[  Y,\left[  X,Y\right]  \right]  -\frac{1}{2}d_{3}%
d_{4}\left[  \left[  X,Y\right]  ,Y\right]  -\frac{1}{4}d_{1}d_{3}d_{4}\left[
X,\left[  X,\left[  X,Y\right]  \right]  \right] \\
& -\frac{1}{4}d_{1}d_{3}d_{4}\left[  X,\left[  Y,\left[  X,Y\right]  \right]
\right]  +\frac{1}{2}d_{1}d_{3}d_{4}\left[  X,\left[  \left[  X,Y\right]
,Y\right]  \right]  -\frac{1}{4}d_{1}d_{3}d_{4}\left[  Y,\left[  X,\left[
X,Y\right]  \right]  \right] \\
& -\frac{1}{4}d_{1}d_{3}d_{4}\left[  Y,\left[  Y,\left[  X,Y\right]  \right]
\right]  +\frac{1}{2}d_{1}d_{3}d_{4}\left[  Y,\left[  \left[  X,Y\right]
,Y\right]  \right]  +\frac{1}{4}d_{1}d_{3}d_{4}\left[  \left[  \left[
X,Y\right]  ,X\right]  ,Y\right] \\
& -\frac{1}{4}d_{1}d_{3}d_{4}\left[  \left[  \left[  X,Y\right]  ,Y\right]
,Y\right]  -\frac{1}{4}d_{2}d_{3}d_{4}\left[  X,\left[  X,\left[  X,Y\right]
\right]  \right]  -\frac{1}{4}d_{2}d_{3}d_{4}\left[  X,\left[  Y,\left[
X,Y\right]  \right]  \right] \\
& +\frac{1}{2}d_{2}d_{3}d_{4}\left[  X,\left[  \left[  X,Y\right]  ,Y\right]
\right]  -\frac{1}{4}d_{2}d_{3}d_{4}\left[  Y,\left[  X,\left[  X,Y\right]
\right]  \right]  -\frac{1}{4}d_{2}d_{3}d_{4}\left[  Y,\left[  Y,\left[
X,Y\right]  \right]  \right] \\
& +\frac{1}{2}d_{2}d_{3}d_{4}\left[  Y,\left[  \left[  X,Y\right]  ,Y\right]
\right]  +\frac{1}{4}d_{2}d_{3}d_{4}\left[  \left[  \left[  X,Y\right]
,X\right]  ,Y\right]  -\frac{1}{4}d_{2}d_{3}d_{4}\left[  \left[  \left[
X,Y\right]  ,Y\right]  ,Y\right] \\
& +\frac{1}{5}d_{1}d_{2}d_{3}d_{4}\left[  X,\left[  X,\left[  X,\left[
X,Y\right]  \right]  \right]  \right]  +\frac{1}{5}d_{1}d_{2}d_{3}d_{4}\left[
X,\left[  X,\left[  Y,\left[  X,Y\right]  \right]  \right]  \right] \\
& -\frac{1}{2}d_{1}d_{2}d_{3}d_{4}\left[  X,\left[  X,\left[  \left[
X,Y\right]  ,Y\right]  \right]  \right]  +\frac{1}{5}d_{1}d_{2}d_{3}%
d_{4}\left[  X,\left[  Y,\left[  X,\left[  X,Y\right]  \right]  \right]
\right] \\
& +\frac{1}{5}d_{1}d_{2}d_{3}d_{4}\left[  X,\left[  Y,\left[  Y,\left[
X,Y\right]  \right]  \right]  \right]  -\frac{1}{2}d_{1}d_{2}d_{3}d_{4}\left[
X,\left[  Y,\left[  \left[  X,Y\right]  ,Y\right]  \right]  \right] \\
& -\frac{1}{3}d_{1}d_{2}d_{3}d_{4}\left[  X,\left[  \left[  \left[
X,Y\right]  ,X\right]  ,Y\right]  \right]  +\frac{1}{3}d_{1}d_{2}d_{3}%
d_{4}\left[  X,\left[  \left[  \left[  X,Y\right]  ,Y\right]  ,Y\right]
\right]
\end{align*}
\begin{align}
& +\frac{1}{5}d_{1}d_{2}d_{3}d_{4}\left[  Y,\left[  X,\left[  X,\left[
X,Y\right]  \right]  \right]  \right]  +\frac{1}{5}d_{1}d_{2}d_{3}d_{4}\left[
Y,\left[  X,\left[  Y,\left[  X,Y\right]  \right]  \right]  \right]
\nonumber\\
& -\frac{1}{2}d_{1}d_{2}d_{3}d_{4}\left[  Y,\left[  X,\left[  \left[
X,Y\right]  ,Y\right]  \right]  \right]  +\frac{1}{5}d_{1}d_{2}d_{3}%
d_{4}\left[  Y,\left[  Y,\left[  X,\left[  X,Y\right]  \right]  \right]
\right] \nonumber\\
& +\frac{1}{5}d_{1}d_{2}d_{3}d_{4}\left[  Y,\left[  Y,\left[  Y,\left[
X,Y\right]  \right]  \right]  \right]  -\frac{1}{2}d_{1}d_{2}d_{3}d_{4}\left[
Y,\left[  Y,\left[  \left[  X,Y\right]  ,Y\right]  \right]  \right]
\nonumber\\
& -\frac{1}{3}d_{1}d_{2}d_{3}d_{4}\left[  Y,\left[  \left[  \left[
X,Y\right]  ,X\right]  ,Y\right]  \right]  +\frac{1}{3}d_{1}d_{2}d_{3}%
d_{4}\left[  Y,\left[  \left[  \left[  X,Y\right]  ,Y\right]  ,Y\right]
\right] \nonumber\\
& -\frac{1}{2}d_{1}d_{2}d_{3}d_{4}\left[  \left[  X,Y\right]  ,\left[
X,\left[  X,Y\right]  \right]  \right]  -\frac{1}{2}d_{1}d_{2}d_{3}%
d_{4}\left[  \left[  X,Y\right]  ,\left[  Y,\left[  X,Y\right]  \right]
\right] \nonumber\\
& +d_{1}d_{2}d_{3}d_{4}\left[  \left[  X,Y\right]  ,\left[  \left[
X,Y\right]  ,Y\right]  \right]  -\frac{1}{8}d_{1}d_{2}d_{3}d_{4}\left[
\left[  X,\left[  \left[  X,Y\right]  ,X\right]  \right]  ,Y\right]
\nonumber\\
& -\frac{1}{8}d_{1}d_{2}d_{3}d_{4}\left[  \left[  X,\left[  \left[
X,Y\right]  ,Y\right]  \right]  ,Y\right]  -\frac{1}{8}d_{1}d_{2}d_{3}%
d_{4}\left[  \left[  Y,\left[  \left[  X,Y\right]  ,X\right]  \right]
,Y\right] \nonumber\\
& -\frac{1}{8}d_{1}d_{2}d_{3}d_{4}\left[  \left[  Y,\left[  \left[
X,Y\right]  ,Y\right]  \right]  ,Y\right]  -\frac{1}{3}d_{1}d_{2}d_{3}%
d_{4}\left[  \left[  \left[  X,Y\right]  ,X\right]  ,\left[  X,Y\right]
\right] \nonumber\\
& +\frac{1}{3}d_{1}d_{2}d_{3}d_{4}\left[  \left[  \left[  X,Y\right]
,Y\right]  ,\left[  X,Y\right]  \right]  -\frac{1}{8}d_{1}d_{2}d_{3}%
d_{4}\left[  \left[  \left[  \left[  X,Y\right]  ,X\right]  ,X\right]
,Y\right] \nonumber\\
& +\frac{1}{8}d_{1}d_{2}d_{3}d_{4}\left[  \left[  \left[  \left[  X,Y\right]
,X\right]  ,Y\right]  ,Y\right]  +\frac{1}{8}d_{1}d_{2}d_{3}d_{4}\left[
\left[  \left[  \left[  X,Y\right]  ,Y\right]  ,X\right]  ,Y\right]
\nonumber\\
& -\frac{1}{8}d_{1}d_{2}d_{3}d_{4}\left[  \left[  \left[  \left[  X,Y\right]
,Y\right]  ,Y\right]  ,Y\right] \label{4.2}%
\end{align}

Letting $n_{51}$ be the right-hand side of (\ref{4.2}) with the first term $Y
$ deleted, we have
\begin{align}
& \left(  \ref{4.1}\right) \nonumber\\
& =\exp\;d_{5}X.\nonumber\\
& \exp\;\left(
\begin{array}
[c]{c}%
\left(  d_{1}+d_{2}+d_{3}+d_{4}\right)  \left(  X+Y\right)  +\frac{1}%
{2}\left(  d_{1}+d_{2}+d_{3}+d_{4}\right)  ^{2}\left[  X,Y\right] \\
+\frac{1}{12}\left(  d_{1}+d_{2}+d_{3}+d_{4}\right)  ^{3}\left[  \left[
X,Y\right]  ,Y-X\right] \\
+\frac{1}{96}\left(  d_{1}+d_{2}+d_{3}+d_{4}\right)  ^{4}\left(  \left[
X+Y,\left[  \left[  X,Y\right]  ,X+Y\right]  \right]  +\left[  \left[  \left[
X,Y\right]  ,X-Y\right]  ,X-Y\right]  \right)
\end{array}
\right)  .\nonumber\\
& \exp\;d_{5}X+d_{5}n_{51}.\exp\;-d_{5}n_{51}\nonumber\\
& \left)  \text{By Proposition \ref{t2.2}}\right( \nonumber\\
& =\exp\;d_{5}X.\nonumber\\
& \exp\;\left(
\begin{array}
[c]{c}%
\left(  d_{1}+d_{2}+d_{3}+d_{4}\right)  \left(  X+Y\right)  +d_{5}Y+\frac
{1}{2}\left(  d_{1}+d_{2}+d_{3}+d_{4}\right)  ^{2}\left[  X,Y\right] \\
+\frac{1}{12}\left(  d_{1}+d_{2}+d_{3}+d_{4}\right)  ^{3}\left[  \left[
X,Y\right]  ,Y-X\right] \\
+\frac{1}{96}\left(  d_{1}+d_{2}+d_{3}+d_{4}\right)  ^{4}\left(  \left[
X+Y,\left[  \left[  X,Y\right]  ,X+Y\right]  \right]  +\left[  \left[  \left[
X,Y\right]  ,X-Y\right]  ,X-Y\right]  \right)
\end{array}
\right)  .\nonumber\\
& \exp\;-d_{5}n_{51}\nonumber\\
& \left)  \text{By (\ref{4.2})}\right( \label{4.3}%
\end{align}

We let $i_{51}$ be the result of $-n_{51}$ by deleting all the terms whose
coefficients contain $d_{1}d_{2}d_{3}d_{4}$. Then, by dint of Theorem
\ref{t2.1}, we have
\begin{align*}
& \delta^{\mathrm{left}}\left(  \exp\right)  \left(
\begin{array}
[c]{c}%
\left(  d_{1}+d_{2}+d_{3}+d_{4}\right)  \left(  X+Y\right)  +d_{5}Y+\frac
{1}{2}\left(  d_{1}+d_{2}+d_{3}+d_{4}\right)  ^{2}\left[  X,Y\right] \\
+\frac{1}{12}\left(  d_{1}+d_{2}+d_{3}+d_{4}\right)  ^{3}\left[  \left[
X,Y\right]  ,Y-X\right] \\
+\frac{1}{96}\left(  d_{1}+d_{2}+d_{3}+d_{4}\right)  ^{4}\left[  X+Y,\left[
\left[  X,Y\right]  ,X+Y\right]  \right] \\
+\frac{1}{96}\left(  d_{1}+d_{2}+d_{3}+d_{4}\right)  ^{4}\left[  \left[
\left[  X,Y\right]  ,X-Y\right]  ,X-Y\right]
\end{array}
\right)  \left(  i_{51}\right) \\
& =\frac{1}{2}d_{1}\left[  X,Y\right]  +\frac{1}{2}d_{2}\left[  X,Y\right]
-\frac{5}{6}d_{1}d_{2}\left[  X,\left[  X,Y\right]  \right]  -\frac{5}{6}%
d_{1}d_{2}\left[  Y,\left[  X,Y\right]  \right]  +\frac{1}{2}d_{1}d_{2}\left[
\left[  X,Y\right]  ,Y\right] \\
& +\frac{1}{2}d_{3}\left[  X,Y\right]  -\frac{5}{6}d_{1}d_{3}\left[  X,\left[
X,Y\right]  \right]  -\frac{5}{6}d_{1}d_{3}\left[  Y,\left[  X,Y\right]
\right]  +\frac{1}{2}d_{1}d_{3}\left[  \left[  X,Y\right]  ,Y\right] \\
& -\frac{5}{6}d_{2}d_{3}\left[  X,\left[  X,Y\right]  \right]  -\frac{5}%
{6}d_{2}d_{3}\left[  Y,\left[  X,Y\right]  \right]  +\frac{1}{2}d_{2}%
d_{3}\left[  \left[  X,Y\right]  ,Y\right]  +\frac{5}{4}d_{1}d_{2}d_{3}\left[
X,\left[  X,\left[  X,Y\right]  \right]  \right] \\
& +\frac{5}{4}d_{1}d_{2}d_{3}\left[  X,\left[  Y,\left[  X,Y\right]  \right]
\right]  -\frac{5}{4}d_{1}d_{2}d_{3}\left[  X,\left[  \left[  X,Y\right]
,Y\right]  \right]  +\frac{5}{4}d_{1}d_{2}d_{3}\left[  Y,\left[  X,\left[
X,Y\right]  \right]  \right] \\
& +\frac{5}{4}d_{1}d_{2}d_{3}\left[  Y,\left[  Y,\left[  X,Y\right]  \right]
\right]  -\frac{5}{4}d_{1}d_{2}d_{3}\left[  Y,\left[  \left[  X,Y\right]
,Y\right]  \right]  -\frac{1}{4}d_{1}d_{2}d_{3}\left[  \left[  \left[
X,Y\right]  ,X\right]  ,Y\right] \\
& +\frac{1}{4}d_{1}d_{2}d_{3}\left[  \left[  \left[  X,Y\right]  ,Y\right]
,Y\right]  +\frac{1}{2}d_{4}\left[  X,Y\right]  -\frac{5}{6}d_{1}d_{4}\left[
X,\left[  X,Y\right]  \right]  -\frac{5}{6}d_{1}d_{4}\left[  Y,\left[
X,Y\right]  \right] \\
& +\frac{1}{2}d_{1}d_{4}\left[  \left[  X,Y\right]  ,Y\right]  -\frac{5}%
{6}d_{2}d_{4}\left[  X,\left[  X,Y\right]  \right]  -\frac{5}{6}d_{2}%
d_{4}\left[  Y,\left[  X,Y\right]  \right]  +\frac{1}{2}d_{2}d_{4}\left[
\left[  X,Y\right]  ,Y\right] \\
& +\frac{5}{4}d_{1}d_{2}d_{4}\left[  X,\left[  X,\left[  X,Y\right]  \right]
\right]  +\frac{5}{4}d_{1}d_{2}d_{4}\left[  X,\left[  Y,\left[  X,Y\right]
\right]  \right]  -\frac{5}{4}d_{1}d_{2}d_{4}\left[  X,\left[  \left[
X,Y\right]  ,Y\right]  \right] \\
& +\frac{5}{4}d_{1}d_{2}d_{4}\left[  Y,\left[  X,\left[  X,Y\right]  \right]
\right]  +\frac{5}{4}d_{1}d_{2}d_{4}\left[  Y,\left[  Y,\left[  X,Y\right]
\right]  \right]  -\frac{5}{4}d_{1}d_{2}d_{4}\left[  Y,\left[  \left[
X,Y\right]  ,Y\right]  \right] \\
& -\frac{1}{4}d_{1}d_{2}d_{4}\left[  \left[  \left[  X,Y\right]  ,X\right]
,Y\right]  +\frac{1}{4}d_{1}d_{2}d_{4}\left[  \left[  \left[  X,Y\right]
,Y\right]  ,Y\right]  -\frac{5}{6}d_{3}d_{4}\left[  X,\left[  X,Y\right]
\right] \\
& -\frac{5}{6}d_{3}d_{4}\left[  Y,\left[  X,Y\right]  \right]  +\frac{1}%
{2}d_{3}d_{4}\left[  \left[  X,Y\right]  ,X\right]  +\frac{5}{4}d_{1}%
d_{3}d_{4}\left[  X,\left[  X,\left[  X,Y\right]  \right]  \right] \\
& +\frac{5}{4}d_{1}d_{3}d_{4}\left[  X,\left[  Y,\left[  X,Y\right]  \right]
\right]  -\frac{5}{4}d_{1}d_{3}d_{4}\left[  X,\left[  \left[  X,Y\right]
,Y\right]  \right]  +\frac{5}{4}d_{1}d_{3}d_{4}\left[  Y,\left[  X,\left[
X,Y\right]  \right]  \right] \\
& +\frac{5}{4}d_{1}d_{3}d_{4}\left[  Y,\left[  Y,\left[  X,Y\right]  \right]
\right]  -\frac{5}{4}d_{1}d_{3}d_{4}\left[  Y,\left[  \left[  X,Y\right]
,Y\right]  \right]  -\frac{1}{4}d_{1}d_{3}d_{4}\left[  \left[  \left[
X,Y\right]  ,X\right]  ,Y\right] \\
& +\frac{1}{4}d_{1}d_{3}d_{4}\left[  \left[  \left[  X,Y\right]  ,Y\right]
,Y\right]  +\frac{5}{4}d_{2}d_{3}d_{4}\left[  X,\left[  X,\left[  X,Y\right]
\right]  \right]  +\frac{5}{4}d_{2}d_{3}d_{4}\left[  X,\left[  Y,\left[
X,Y\right]  \right]  \right] \\
& -\frac{5}{4}d_{2}d_{3}d_{4}\left[  X,\left[  \left[  X,Y\right]  ,Y\right]
\right]  +\frac{5}{4}d_{2}d_{3}d_{4}\left[  Y,\left[  X,\left[  X,Y\right]
\right]  \right]  +\frac{5}{4}d_{2}d_{3}d_{4}\left[  Y,\left[  Y,\left[
X,Y\right]  \right]  \right] \\
& -\frac{5}{4}d_{2}d_{3}d_{4}\left[  Y,\left[  \left[  X,Y\right]  ,Y\right]
\right]  -\frac{1}{4}d_{2}d_{3}d_{4}\left[  \left[  \left[  X,Y\right]
,X\right]  ,Y\right]  +\frac{1}{4}d_{2}d_{3}d_{4}\left[  \left[  \left[
X,Y\right]  ,Y\right]  ,Y\right] \\
& -\frac{5}{3}d_{1}d_{2}d_{3}d_{4}\left[  X,\left[  X,\left[  X,\left[
X,Y\right]  \right]  \right]  \right]  -\frac{5}{3}d_{1}d_{2}d_{3}d_{4}\left[
X,\left[  X,\left[  Y,\left[  X,Y\right]  \right]  \right]  \right] \\
& +2d_{1}d_{2}d_{3}d_{4}\left[  X,\left[  X,\left[  \left[  X,Y\right]
,Y\right]  \right]  \right]  -\frac{5}{3}d_{1}d_{2}d_{3}d_{4}\left[  X,\left[
Y,\left[  X,\left[  X,Y\right]  \right]  \right]  \right] \\
& -\frac{5}{3}d_{1}d_{2}d_{3}d_{4}\left[  X,\left[  Y,\left[  Y,\left[
X,Y\right]  \right]  \right]  \right]  +2d_{1}d_{2}d_{3}d_{4}\left[  X,\left[
Y,\left[  \left[  X,Y\right]  ,Y\right]  \right]  \right] \\
& +\frac{1}{2}d_{1}d_{2}d_{3}d_{4}\left[  X,\left[  \left[  \left[
X,Y\right]  ,X\right]  ,Y\right]  \right]  -\frac{1}{2}d_{1}d_{2}d_{3}%
d_{4}\left[  X,\left[  \left[  \left[  X,Y\right]  ,Y\right]  ,Y\right]
\right]
\end{align*}
\begin{align}
& -\frac{5}{3}d_{1}d_{2}d_{3}d_{4}\left[  Y,\left[  X,\left[  X,\left[
X,Y\right]  \right]  \right]  \right]  -\frac{5}{3}d_{1}d_{2}d_{3}d_{4}\left[
Y,\left[  X,\left[  Y,\left[  X,Y\right]  \right]  \right]  \right]
\nonumber\\
& +2d_{1}d_{2}d_{3}d_{4}\left[  Y,\left[  X,\left[  \left[  X,Y\right]
,Y\right]  \right]  \right]  -\frac{5}{3}d_{1}d_{2}d_{3}d_{4}\left[  Y,\left[
Y,\left[  X,\left[  X,Y\right]  \right]  \right]  \right] \nonumber\\
& -\frac{5}{3}d_{1}d_{2}d_{3}d_{4}\left[  Y,\left[  Y,\left[  Y,\left[
X,Y\right]  \right]  \right]  \right]  +2d_{1}d_{2}d_{3}d_{4}\left[  Y,\left[
Y,\left[  \left[  X,Y\right]  ,Y\right]  \right]  \right] \nonumber\\
& +\frac{1}{2}d_{1}d_{2}d_{3}d_{4}\left[  Y,\left[  \left[  \left[
X,Y\right]  ,X\right]  ,Y\right]  \right]  -\frac{1}{2}d_{1}d_{2}d_{3}%
d_{4}\left[  Y,\left[  \left[  \left[  X,Y\right]  ,Y\right]  ,Y\right]
\right] \nonumber\\
& +2d_{1}d_{2}d_{3}d_{4}\left[  \left[  X,Y\right]  ,\left[  X,\left[
X,Y\right]  \right]  \right]  +2d_{1}d_{2}d_{3}d_{4}\left[  \left[
X,Y\right]  ,\left[  Y,\left[  X,Y\right]  \right]  \right] \nonumber\\
& -\frac{3}{2}d_{1}d_{2}d_{3}d_{4}\left[  \left[  X,Y\right]  ,\left[  \left[
X,Y\right]  ,Y\right]  \right]  +\frac{1}{2}d_{1}d_{2}d_{3}d_{4}\left[
\left[  \left[  X,Y\right]  ,X\right]  ,\left[  X,Y\right]  \right]
\nonumber\\
& -\frac{1}{2}d_{1}d_{2}d_{3}d_{4}\left[  \left[  \left[  X,Y\right]
,Y\right]  ,\left[  X,Y\right]  \right] \label{4.4}%
\end{align}

Letting $n_{52}$ be the right-hand side of (\ref{4.4}), we have
\begin{align}
& \left(  \ref{4.3}\right) \nonumber\\
& =\exp\;d_{5}X.\nonumber\\
& \exp\;\left(
\begin{array}
[c]{c}%
\left(  d_{1}+d_{2}+d_{3}+d_{4}\right)  \left(  X+Y\right)  +d_{5}Y+\frac
{1}{2}\left(  d_{1}+d_{2}+d_{3}+d_{4}\right)  ^{2}\left[  X,Y\right] \\
+\frac{1}{12}\left(  d_{1}+d_{2}+d_{3}+d_{4}\right)  ^{3}\left[  \left[
X,Y\right]  ,Y-X\right] \\
+\frac{1}{96}\left(  d_{1}+d_{2}+d_{3}+d_{4}\right)  ^{4}\left(  \left[
X+Y,\left[  \left[  X,Y\right]  ,X+Y\right]  \right]  +\left[  \left[  \left[
X,Y\right]  ,X-Y\right]  ,X-Y\right]  \right)
\end{array}
\right)  .\nonumber\\
& \exp\;d_{5}n_{52}.\exp\;-d_{5}n_{51}-d_{5}n_{52}\nonumber\\
& \left)  \text{By Proposition \ref{t2.2}}\right( \nonumber\\
& =\exp\;d_{5}X.\nonumber\\
& \exp\;\left(
\begin{array}
[c]{c}%
\left(  d_{1}+d_{2}+d_{3}+d_{4}\right)  \left(  X+Y\right)  +d_{5}Y+\frac
{1}{2}\left(  d_{1}+d_{2}+d_{3}+d_{4}\right)  ^{2}\left[  X,Y\right] \\
+\frac{1}{12}\left(  d_{1}+d_{2}+d_{3}+d_{4}\right)  ^{3}\left[  \left[
X,Y\right]  ,Y-X\right] \\
+\frac{1}{96}\left(  d_{1}+d_{2}+d_{3}+d_{4}\right)  ^{4}\left(  \left[
X+Y,\left[  \left[  X,Y\right]  ,X+Y\right]  \right]  +\left[  \left[  \left[
X,Y\right]  ,X-Y\right]  ,X-Y\right]  \right)
\end{array}
\right)  .\nonumber\\
& \left(  n_{52}\right)  _{d_{5}}.\exp\;-d_{5}n_{51}-d_{5}n_{52}\nonumber\\
& \left)  \text{By Proposition \ref{t2.3}}\right( \nonumber\\
& =\exp\;d_{5}X.\nonumber\\
& \exp\;\left(
\begin{array}
[c]{c}%
\left(  d_{1}+d_{2}+d_{3}+d_{4}\right)  \left(  X+Y\right)  +d_{5}Y+\frac
{1}{2}\left(  d_{1}+d_{2}+d_{3}+d_{4}\right)  ^{2}\left[  X,Y\right] \\
+\frac{1}{12}\left(  d_{1}+d_{2}+d_{3}+d_{4}\right)  ^{3}\left[  \left[
X,Y\right]  ,Y-X\right] \\
+\frac{1}{96}\left(  d_{1}+d_{2}+d_{3}+d_{4}\right)  ^{4}\left(  \left[
X+Y,\left[  \left[  X,Y\right]  ,X+Y\right]  \right]  +\left[  \left[  \left[
X,Y\right]  ,X-Y\right]  ,X-Y\right]  \right) \\
+d_{5}i_{51}%
\end{array}
\right)  .\nonumber\\
& \exp\;-d_{5}n_{51}-d_{5}n_{52}\nonumber\\
& \left)  \text{By (\ref{4.4})}\right( \label{4.5}%
\end{align}

We let $i_{52}$ be the result of $-n_{51}-n_{52}$ by deleting all the terms
whose coefficients contain $d_{1}d_{2}d_{3}d_{4}$. Then, thanks to Theorem
\ref{t2.1}, we have
\begin{align}
& \delta^{\mathrm{left}}\left(  \exp\right)  \left(
\begin{array}
[c]{c}%
\left(  d_{1}+d_{2}+d_{3}+d_{4}\right)  \left(  X+Y\right)  +d_{5}Y+\frac
{1}{2}\left(  d_{1}+d_{2}+d_{3}+d_{4}\right)  ^{2}\left[  X,Y\right] \\
+\frac{1}{12}\left(  d_{1}+d_{2}+d_{3}+d_{4}\right)  ^{3}\left[  \left[
X,Y\right]  ,Y-X\right] \\
+\frac{1}{96}\left(  d_{1}+d_{2}+d_{3}+d_{4}\right)  ^{4}\left[  X+Y,\left[
\left[  X,Y\right]  ,X+Y\right]  \right] \\
+\frac{1}{96}\left(  d_{1}+d_{2}+d_{3}+d_{4}\right)  ^{4}\left[  \left[
\left[  X,Y\right]  ,X-Y\right]  ,X-Y\right]  +d_{5}i_{51}%
\end{array}
\right)  \left(  i_{52}\right) \nonumber\\
& =\frac{1}{2}d_{1}d_{2}\left[  X,\left[  X,Y\right]  \right]  +\frac{1}%
{2}d_{1}d_{2}\left[  Y,\left[  X,Y\right]  \right]  +\frac{1}{2}d_{1}%
d_{3}\left[  X,\left[  X,Y\right]  \right]  +\frac{1}{2}d_{1}d_{3}\left[
Y,\left[  X,Y\right]  \right] \nonumber\\
& +\frac{1}{2}d_{2}d_{3}\left[  X,\left[  X,Y\right]  \right]  +\frac{1}%
{2}d_{2}d_{3}\left[  Y,\left[  X,Y\right]  \right]  -\frac{7}{4}d_{1}%
d_{2}d_{3}\left[  X,\left[  X,\left[  X,Y\right]  \right]  \right] \nonumber\\
& -\frac{7}{4}d_{1}d_{2}d_{3}\left[  X,\left[  Y,\left[  X,Y\right]  \right]
\right]  +\frac{3}{4}d_{1}d_{2}d_{3}\left[  X,\left[  \left[  X,Y\right]
,Y\right]  \right]  -\frac{7}{4}d_{1}d_{2}d_{3}\left[  Y,\left[  X,\left[
X,Y\right]  \right]  \right] \nonumber\\
& -\frac{7}{4}d_{1}d_{2}d_{3}\left[  Y,\left[  Y,\left[  X,Y\right]  \right]
\right]  +\frac{3}{4}d_{1}d_{2}d_{3}\left[  Y,\left[  \left[  X,Y\right]
,Y\right]  \right]  +\frac{1}{2}d_{1}d_{4}\left[  X,\left[  X,Y\right]
\right] \nonumber\\
& +\frac{1}{2}d_{1}d_{4}\left[  Y,\left[  X,Y\right]  \right]  +\frac{1}%
{2}d_{2}d_{4}\left[  X,\left[  X,Y\right]  \right]  +\frac{1}{2}d_{2}%
d_{4}\left[  Y,\left[  X,Y\right]  \right]  -\frac{7}{4}d_{1}d_{2}d_{4}\left[
X,\left[  X,\left[  X,Y\right]  \right]  \right] \nonumber\\
& -\frac{7}{4}d_{1}d_{2}d_{4}\left[  X,\left[  Y,\left[  X,Y\right]  \right]
\right]  +\frac{3}{4}d_{1}d_{2}d_{4}\left[  X,\left[  \left[  X,Y\right]
,Y\right]  \right]  -\frac{7}{4}d_{1}d_{2}d_{4}\left[  Y,\left[  X,\left[
X,Y\right]  \right]  \right] \nonumber\\
& -\frac{7}{4}d_{1}d_{2}d_{4}\left[  Y,\left[  Y,\left[  X,Y\right]  \right]
\right]  +\frac{3}{4}d_{1}d_{2}d_{4}\left[  Y,\left[  \left[  X,Y\right]
,Y\right]  \right]  +\frac{1}{2}d_{3}d_{4}\left[  X,\left[  X,Y\right]
\right] \nonumber\\
& +\frac{1}{2}d_{3}d_{4}\left[  Y,\left[  X,Y\right]  \right]  -\frac{7}%
{4}d_{1}d_{3}d_{4}\left[  X,\left[  X,\left[  X,Y\right]  \right]  \right]
-\frac{7}{4}d_{1}d_{3}d_{4}\left[  X,\left[  Y,\left[  X,Y\right]  \right]
\right] \nonumber\\
& +\frac{3}{4}d_{1}d_{3}d_{4}\left[  X,\left[  \left[  X,Y\right]  ,Y\right]
\right]  -\frac{7}{4}d_{1}d_{3}d_{4}\left[  Y,\left[  X,\left[  X,Y\right]
\right]  \right]  -\frac{7}{4}d_{1}d_{3}d_{4}\left[  Y,\left[  Y,\left[
X,Y\right]  \right]  \right] \nonumber\\
& +\frac{3}{4}d_{1}d_{3}d_{4}\left[  Y,\left[  \left[  X,Y\right]  ,Y\right]
\right]  -\frac{7}{4}d_{2}d_{3}d_{4}\left[  X,\left[  X,\left[  X,Y\right]
\right]  \right]  -\frac{7}{4}d_{2}d_{3}d_{4}\left[  X,\left[  Y,\left[
X,Y\right]  \right]  \right] \nonumber\\
& +\frac{3}{4}d_{2}d_{3}d_{4}\left[  X,\left[  \left[  X,Y\right]  ,Y\right]
\right]  -\frac{7}{4}d_{2}d_{3}d_{4}\left[  Y,\left[  X,\left[  X,Y\right]
\right]  \right]  -\frac{7}{4}d_{2}d_{3}d_{4}\left[  Y,\left[  Y,\left[
X,Y\right]  \right]  \right] \nonumber\\
& +\frac{3}{4}d_{2}d_{3}d_{4}\left[  Y,\left[  \left[  X,Y\right]  ,Y\right]
\right]  +3d_{1}d_{2}d_{3}d_{4}\left[  X,\left[  X,\left[  X,\left[
X,Y\right]  \right]  \right]  \right] \nonumber\\
& +3d_{1}d_{2}d_{3}d_{4}\left[  X,\left[  X,\left[  Y,\left[  X,Y\right]
\right]  \right]  \right]  +\frac{1}{2}d_{1}d_{2}d_{3}d_{4}\left[  X,\left[
X,\left[  \left[  X,Y\right]  ,Y\right]  \right]  \right] \nonumber\\
& +3d_{1}d_{2}d_{3}d_{4}\left[  X,\left[  Y,\left[  X,\left[  X,Y\right]
\right]  \right]  \right]  +3d_{1}d_{2}d_{3}d_{4}\left[  X,\left[  Y,\left[
Y,\left[  X,Y\right]  \right]  \right]  \right] \nonumber\\
& -\frac{1}{2}d_{1}d_{2}d_{3}d_{4}\left[  X,\left[  Y,\left[  \left[
X,Y\right]  ,Y\right]  \right]  \right]  +3d_{1}d_{2}d_{3}d_{4}\left[
Y,\left[  X,\left[  X,\left[  X,Y\right]  \right]  \right]  \right]
\nonumber\\
& +3d_{1}d_{2}d_{3}d_{4}\left[  Y,\left[  X,\left[  Y,\left[  X,Y\right]
\right]  \right]  \right]  -\frac{1}{2}d_{1}d_{2}d_{3}d_{4}\left[  Y,\left[
X,\left[  \left[  X,Y\right]  ,Y\right]  \right]  \right] \nonumber\\
& +3d_{1}d_{2}d_{3}d_{4}\left[  Y,\left[  Y,\left[  X,\left[  X,Y\right]
\right]  \right]  \right]  +3d_{1}d_{2}d_{3}d_{4}\left[  Y,\left[  Y,\left[
Y,\left[  X,Y\right]  \right]  \right]  \right] \nonumber\\
& -\frac{1}{2}d_{1}d_{2}d_{3}d_{4}\left[  Y,\left[  Y,\left[  \left[
X,Y\right]  ,Y\right]  \right]  \right]  -\frac{3}{2}d_{1}d_{2}d_{3}%
d_{4}\left[  \left[  X,Y\right]  ,\left[  X,\left[  X,Y\right]  \right]
\right] \nonumber\\
& -\frac{3}{2}d_{1}d_{2}d_{3}d_{4}\left[  \left[  X,Y\right]  ,\left[
Y,\left[  X,Y\right]  \right]  \right] \label{4.6}%
\end{align}

Letting $n_{53}$ be the right-hand side of (\ref{4.6}), we have
\begin{align}
& \left(  \ref{4.5}\right) \nonumber\\
& =\exp\;d_{5}X.\nonumber\\
& \exp\;\left(
\begin{array}
[c]{c}%
\left(  d_{1}+d_{2}+d_{3}+d_{4}\right)  \left(  X+Y\right)  +d_{5}Y+\frac
{1}{2}\left(  d_{1}+d_{2}+d_{3}+d_{4}\right)  ^{2}\left[  X,Y\right] \\
+\frac{1}{12}\left(  d_{1}+d_{2}+d_{3}+d_{4}\right)  ^{3}\left[  \left[
X,Y\right]  ,Y-X\right] \\
+\frac{1}{96}\left(  d_{1}+d_{2}+d_{3}+d_{4}\right)  ^{4}\left(  \left[
X+Y,\left[  \left[  X,Y\right]  ,X+Y\right]  \right]  +\left[  \left[  \left[
X,Y\right]  ,X-Y\right]  ,X-Y\right]  \right) \\
+d_{5}i_{51}%
\end{array}
\right)  .\nonumber\\
& \exp\;d_{5}n_{53}.\exp\;-d_{5}n_{51}-d_{5}n_{52}-d_{5}n_{53}\nonumber\\
& \left)  \text{By Proposition \ref{t2.2}}\right( \nonumber\\
& =\exp\;d_{5}X.\nonumber\\
& \exp\;\left(
\begin{array}
[c]{c}%
\left(  d_{1}+d_{2}+d_{3}+d_{4}\right)  \left(  X+Y\right)  +d_{5}Y+\frac
{1}{2}\left(  d_{1}+d_{2}+d_{3}+d_{4}\right)  ^{2}\left[  X,Y\right] \\
+\frac{1}{12}\left(  d_{1}+d_{2}+d_{3}+d_{4}\right)  ^{3}\left[  \left[
X,Y\right]  ,Y-X\right] \\
+\frac{1}{96}\left(  d_{1}+d_{2}+d_{3}+d_{4}\right)  ^{4}\left(  \left[
X+Y,\left[  \left[  X,Y\right]  ,X+Y\right]  \right]  +\left[  \left[  \left[
X,Y\right]  ,X-Y\right]  ,X-Y\right]  \right) \\
+d_{5}i_{51}%
\end{array}
\right)  .\nonumber\\
& \left(  n_{53}\right)  _{d_{5}}.\exp\;-d_{5}n_{51}-d_{5}n_{52}-d_{5}%
n_{53}\nonumber\\
& \left)  \text{By Proposition \ref{t2.3}}\right( \nonumber\\
& =\exp\;d_{5}X.\nonumber\\
& \exp\;\left(
\begin{array}
[c]{c}%
\left(  d_{1}+d_{2}+d_{3}+d_{4}\right)  \left(  X+Y\right)  +d_{5}Y+\frac
{1}{2}\left(  d_{1}+d_{2}+d_{3}+d_{4}\right)  ^{2}\left[  X,Y\right] \\
+\frac{1}{12}\left(  d_{1}+d_{2}+d_{3}+d_{4}\right)  ^{3}\left[  \left[
X,Y\right]  ,Y-X\right] \\
+\frac{1}{96}\left(  d_{1}+d_{2}+d_{3}+d_{4}\right)  ^{4}\left(  \left[
X+Y,\left[  \left[  X,Y\right]  ,X+Y\right]  \right]  +\left[  \left[  \left[
X,Y\right]  ,X-Y\right]  ,X-Y\right]  \right) \\
+d_{5}i_{51}+d_{5}i_{52}%
\end{array}
\right)  .\nonumber\\
& \exp\;-d_{5}n_{51}-d_{5}n_{52}-d_{5}n_{53}\nonumber\\
& \left)  \text{By (\ref{4.6})}\right( \label{4.7}%
\end{align}

We let $i_{53}$ be the result of $-n_{51}-n_{52}-n_{53}$ by deleting all the
terms whose coefficients contain $d_{1}d_{2}d_{3}d_{4}$. Then, due to Theorem
\ref{t2.1}, we have
\begin{align}
& \delta^{\mathrm{left}}\left(  \exp\right)  \left(
\begin{array}
[c]{c}%
\left(  d_{1}+d_{2}+d_{3}+d_{4}\right)  \left(  X+Y\right)  +d_{5}Y+\frac
{1}{2}\left(  d_{1}+d_{2}+d_{3}+d_{4}\right)  ^{2}\left[  X,Y\right] \\
+\frac{1}{12}\left(  d_{1}+d_{2}+d_{3}+d_{4}\right)  ^{3}\left[  \left[
X,Y\right]  ,Y-X\right] \\
+\frac{1}{96}\left(  d_{1}+d_{2}+d_{3}+d_{4}\right)  ^{4}\left[  X+Y,\left[
\left[  X,Y\right]  ,X+Y\right]  \right] \\
+\frac{1}{96}\left(  d_{1}+d_{2}+d_{3}+d_{4}\right)  ^{4}\left[  \left[
\left[  X,Y\right]  ,X-Y\right]  ,X-Y\right] \\
+d_{5}i_{51}+d_{5}i_{52}%
\end{array}
\right)  \left(  i_{53}\right) \nonumber\\
& =\frac{3}{4}d_{1}d_{2}d_{3}\left[  X,\left[  X,\left[  X,Y\right]  \right]
\right]  +\frac{3}{4}d_{1}d_{2}d_{3}\left[  X,\left[  Y,\left[  X,Y\right]
\right]  \right]  +\frac{3}{4}d_{1}d_{2}d_{3}\left[  Y,\left[  X,\left[
X,Y\right]  \right]  \right] \nonumber\\
& +\frac{3}{4}d_{1}d_{2}d_{3}\left[  Y,\left[  Y,\left[  X,Y\right]  \right]
\right]  +\frac{3}{4}d_{1}d_{2}d_{4}\left[  X,\left[  X,\left[  X,Y\right]
\right]  \right]  +\frac{3}{4}d_{1}d_{2}d_{4}\left[  X,\left[  Y,\left[
X,Y\right]  \right]  \right] \nonumber\\
& +\frac{3}{4}d_{1}d_{2}d_{4}\left[  Y,\left[  X,\left[  X,Y\right]  \right]
\right]  +\frac{3}{4}d_{1}d_{2}d_{4}\left[  Y,\left[  Y,\left[  X,Y\right]
\right]  \right]  +\frac{3}{4}d_{1}d_{3}d_{4}\left[  X,\left[  X,\left[
X,Y\right]  \right]  \right] \nonumber\\
& +\frac{3}{4}d_{1}d_{3}d_{4}\left[  X,\left[  Y,\left[  X,Y\right]  \right]
\right]  +\frac{3}{4}d_{1}d_{3}d_{4}\left[  Y,\left[  X,\left[  X,Y\right]
\right]  \right]  +\frac{3}{4}d_{1}d_{3}d_{4}\left[  Y,\left[  Y,\left[
X,Y\right]  \right]  \right] \nonumber\\
& +\frac{3}{4}d_{2}d_{3}d_{4}\left[  X,\left[  X,\left[  X,Y\right]  \right]
\right]  +\frac{3}{4}d_{2}d_{3}d_{4}\left[  X,\left[  Y,\left[  X,Y\right]
\right]  \right]  +\frac{3}{4}d_{2}d_{3}d_{4}\left[  Y,\left[  X,\left[
X,Y\right]  \right]  \right] \nonumber\\
& +\frac{3}{4}d_{2}d_{3}d_{4}\left[  Y,\left[  Y,\left[  X,Y\right]  \right]
\right]  -\frac{1}{2}d_{1}d_{2}d_{3}d_{4}\left[  X,\left[  X,\left[  X,\left[
X,Y\right]  \right]  \right]  \right] \nonumber\\
& -\frac{1}{2}d_{1}d_{2}d_{3}d_{4}\left[  X,\left[  X,\left[  Y,\left[
X,Y\right]  \right]  \right]  \right]  -\frac{1}{2}d_{1}d_{2}d_{3}d_{4}\left[
X,\left[  Y,\left[  X,\left[  X,Y\right]  \right]  \right]  \right]
\nonumber\\
& -\frac{1}{2}d_{1}d_{2}d_{3}d_{4}\left[  X,\left[  Y,\left[  Y,\left[
X,Y\right]  \right]  \right]  \right]  -\frac{1}{2}d_{1}d_{2}d_{3}d_{4}\left[
Y,\left[  X,\left[  X,\left[  X,Y\right]  \right]  \right]  \right]
\nonumber\\
& -\frac{1}{2}d_{1}d_{2}d_{3}d_{4}\left[  Y,\left[  X,\left[  Y,\left[
X,Y\right]  \right]  \right]  \right]  -\frac{1}{2}d_{1}d_{2}d_{3}d_{4}\left[
Y,\left[  Y,\left[  X,\left[  X,Y\right]  \right]  \right]  \right]
\nonumber\\
& -\frac{1}{2}d_{1}d_{2}d_{3}d_{4}\left[  Y,\left[  Y,\left[  Y,\left[
X,Y\right]  \right]  \right]  \right] \label{4.8}%
\end{align}

Letting $n_{54}$ be the right-hand side of (\ref{4.8}), we have
\begin{align}
& \left(  \ref{4.7}\right) \nonumber\\
& =\exp\;d_{5}X.\nonumber\\
& \exp\;\left(
\begin{array}
[c]{c}%
\left(  d_{1}+d_{2}+d_{3}+d_{4}\right)  \left(  X+Y\right)  +d_{5}Y+\frac
{1}{2}\left(  d_{1}+d_{2}+d_{3}+d_{4}\right)  ^{2}\left[  X,Y\right] \\
+\frac{1}{12}\left(  d_{1}+d_{2}+d_{3}+d_{4}\right)  ^{3}\left[  \left[
X,Y\right]  ,Y-X\right] \\
+\frac{1}{96}\left(  d_{1}+d_{2}+d_{3}+d_{4}\right)  ^{4}\left(  \left[
X+Y,\left[  \left[  X,Y\right]  ,X+Y\right]  \right]  +\left[  \left[  \left[
X,Y\right]  ,X-Y\right]  ,X-Y\right]  \right) \\
+d_{5}i_{51}+d_{5}i_{52}%
\end{array}
\right)  .\nonumber\\
& \exp\;d_{5}n_{54}.\exp\;-d_{5}n_{51}-d_{5}n_{52}-d_{5}n_{53}-d_{5}%
n_{54}\nonumber\\
& \left)  \text{By Proposition \ref{t2.2}}\right( \nonumber\\
& =\exp\;d_{5}X.\nonumber\\
& \exp\;\left(
\begin{array}
[c]{c}%
\left(  d_{1}+d_{2}+d_{3}+d_{4}\right)  \left(  X+Y\right)  +d_{5}Y+\frac
{1}{2}\left(  d_{1}+d_{2}+d_{3}+d_{4}\right)  ^{2}\left[  X,Y\right] \\
+\frac{1}{12}\left(  d_{1}+d_{2}+d_{3}+d_{4}\right)  ^{3}\left[  \left[
X,Y\right]  ,Y-X\right] \\
+\frac{1}{96}\left(  d_{1}+d_{2}+d_{3}+d_{4}\right)  ^{4}\left(  \left[
X+Y,\left[  \left[  X,Y\right]  ,X+Y\right]  \right]  +\left[  \left[  \left[
X,Y\right]  ,X-Y\right]  ,X-Y\right]  \right) \\
+d_{5}i_{51}+d_{5}i_{52}%
\end{array}
\right)  .\nonumber\\
& \left(  n_{54}\right)  _{d_{5}}.\exp\;-d_{5}n_{51}-d_{5}n_{52}-d_{5}%
n_{53}-d_{5}n_{54}\nonumber\\
& \left)  \text{By Proposition \ref{t2.3}}\right( \nonumber\\
& =\exp\;d_{5}X.\nonumber\\
& \exp\;\left(
\begin{array}
[c]{c}%
\left(  d_{1}+d_{2}+d_{3}+d_{4}\right)  \left(  X+Y\right)  +d_{5}Y+\frac
{1}{2}\left(  d_{1}+d_{2}+d_{3}+d_{4}\right)  ^{2}\left[  X,Y\right] \\
+\frac{1}{12}\left(  d_{1}+d_{2}+d_{3}+d_{4}\right)  ^{3}\left[  \left[
X,Y\right]  ,Y-X\right] \\
+\frac{1}{96}\left(  d_{1}+d_{2}+d_{3}+d_{4}\right)  ^{4}\left(  \left[
X+Y,\left[  \left[  X,Y\right]  ,X+Y\right]  \right]  +\left[  \left[  \left[
X,Y\right]  ,X-Y\right]  ,X-Y\right]  \right) \\
+d_{5}i_{51}+d_{5}i_{52}+d_{5}i_{53}%
\end{array}
\right)  .\nonumber\\
& \exp\;-d_{5}n_{51}-d_{5}n_{52}-d_{5}n_{53}-d_{5}n_{54}\nonumber\\
& \left)  \text{By (\ref{4.8})}\right( \label{4.9}%
\end{align}

Since the coefficient of every term in $-n_{51}-n_{52}-n_{53}-n_{54}$ contains
$d_{1}d_{2}d_{3}d_{4}$, we turn our attention to the left $\exp\;d_{5}X$. Now,
thanks to Theorem \ref{t2.1}, we have
\begin{align*}
& \delta^{\mathrm{right}}\left(  \exp\right)  \left(
\begin{array}
[c]{c}%
\left(  d_{1}+d_{2}+d_{3}+d_{4}\right)  \left(  X+Y\right)  +d_{5}Y+\frac
{1}{2}\left(  d_{1}+d_{2}+d_{3}+d_{4}\right)  ^{2}\left[  X,Y\right] \\
+\frac{1}{12}\left(  d_{1}+d_{2}+d_{3}+d_{4}\right)  ^{3}\left[  \left[
X,Y\right]  ,Y-X\right] \\
+\frac{1}{96}\left(  d_{1}+d_{2}+d_{3}+d_{4}\right)  ^{4}\left[  X+Y,\left[
\left[  X,Y\right]  ,X+Y\right]  \right] \\
+\frac{1}{96}\left(  d_{1}+d_{2}+d_{3}+d_{4}\right)  ^{4}\left[  \left[
\left[  X,Y\right]  ,X-Y\right]  ,X-Y\right] \\
+d_{5}i_{51}+d_{5}i_{52}+d_{5}i_{53}%
\end{array}
\right)  \left(  X\right) \\
& =X-\frac{1}{2}d_{1}\left[  X,Y\right]  -\frac{1}{2}d_{2}\left[  X,Y\right]
-\frac{1}{3}d_{1}d_{2}\left[  X,\left[  X,Y\right]  \right]  -\frac{1}{3}%
d_{1}d_{2}\left[  Y,\left[  X,Y\right]  \right] \\
& +\frac{1}{2}d_{1}d_{2}\left[  \left[  X,Y\right]  ,X\right]  -\frac{1}%
{2}d_{3}\left[  X,Y\right]  -\frac{1}{3}d_{1}d_{3}\left[  X,\left[
X,Y\right]  \right]  -\frac{1}{3}d_{1}d_{3}\left[  Y,\left[  X,Y\right]
\right] \\
& +\frac{1}{2}d_{1}d_{3}\left[  \left[  X,Y\right]  ,X\right]  -\frac{1}%
{3}d_{2}d_{3}\left[  X,\left[  X,Y\right]  \right]  -\frac{1}{3}d_{2}%
d_{3}\left[  Y,\left[  X,Y\right]  \right]  +\frac{1}{2}d_{2}d_{3}\left[
\left[  X,Y\right]  ,X\right] \\
& -\frac{1}{4}d_{1}d_{2}d_{3}\left[  X,\left[  X,\left[  X,Y\right]  \right]
\right]  -\frac{1}{4}d_{1}d_{2}d_{3}\left[  X,\left[  Y,\left[  X,Y\right]
\right]  \right]  +\frac{1}{2}d_{1}d_{2}d_{3}\left[  X,\left[  \left[
X,Y\right]  ,X\right]  \right] \\
& -\frac{1}{4}d_{1}d_{2}d_{3}\left[  Y,\left[  X,\left[  X,Y\right]  \right]
\right]  -\frac{1}{4}d_{1}d_{2}d_{3}\left[  Y,\left[  Y,\left[  X,Y\right]
\right]  \right]  +\frac{1}{2}d_{1}d_{2}d_{3}\left[  Y,\left[  \left[
X,Y\right]  ,X\right]  \right] \\
& -\frac{1}{4}d_{1}d_{2}d_{3}\left[  \left[  \left[  X,Y\right]  ,X\right]
,X\right]  +\frac{1}{4}d_{1}d_{2}d_{3}\left[  \left[  \left[  X,Y\right]
,Y\right]  ,X\right]  -\frac{1}{2}d_{4}\left[  X,Y\right]  -\frac{1}{3}%
d_{1}d_{4}\left[  X,\left[  X,Y\right]  \right] \\
& -\frac{1}{3}d_{1}d_{4}\left[  Y,\left[  X,Y\right]  \right]  +\frac{1}%
{2}d_{1}d_{4}\left[  \left[  X,Y\right]  ,X\right]  -\frac{1}{3}d_{2}%
d_{4}\left[  X,\left[  X,Y\right]  \right]  -\frac{1}{3}d_{2}d_{4}\left[
Y,\left[  X,Y\right]  \right] \\
& +\frac{1}{2}d_{2}d_{4}\left[  \left[  X,Y\right]  ,X\right]  -\frac{1}%
{4}d_{1}d_{2}d_{4}\left[  X,\left[  X,\left[  X,Y\right]  \right]  \right]
-\frac{1}{4}d_{1}d_{2}d_{4}\left[  X,\left[  Y,\left[  X,Y\right]  \right]
\right] \\
& +\frac{1}{2}d_{1}d_{2}d_{4}\left[  X,\left[  \left[  X,Y\right]  ,X\right]
\right]  -\frac{1}{4}d_{1}d_{2}d_{4}\left[  Y,\left[  X,\left[  X,Y\right]
\right]  \right]  -\frac{1}{4}d_{1}d_{2}d_{4}\left[  Y,\left[  Y,\left[
X,Y\right]  \right]  \right] \\
& +\frac{1}{2}d_{1}d_{2}d_{4}\left[  Y,\left[  \left[  X,Y\right]  ,X\right]
\right]  -\frac{1}{4}d_{1}d_{2}d_{4}\left[  \left[  \left[  X,Y\right]
,X\right]  ,X\right]  +\frac{1}{4}d_{1}d_{2}d_{4}\left[  \left[  \left[
X,Y\right]  ,Y\right]  ,X\right] \\
& -\frac{1}{3}d_{3}d_{4}\left[  X,\left[  X,Y\right]  \right]  -\frac{1}%
{3}d_{3}d_{4}\left[  Y,\left[  X,Y\right]  \right]  +\frac{1}{2}d_{3}%
d_{4}\left[  \left[  X,Y\right]  ,X\right]  -\frac{1}{4}d_{1}d_{3}d_{4}\left[
X,\left[  X,\left[  X,Y\right]  \right]  \right] \\
& -\frac{1}{4}d_{1}d_{3}d_{4}\left[  X,\left[  Y,\left[  X,Y\right]  \right]
\right]  +\frac{1}{2}d_{1}d_{3}d_{4}\left[  X,\left[  \left[  X,Y\right]
,X\right]  \right]  -\frac{1}{4}d_{1}d_{3}d_{4}\left[  Y,\left[  X,\left[
X,Y\right]  \right]  \right] \\
& -\frac{1}{4}d_{1}d_{3}d_{4}\left[  Y,\left[  Y,\left[  X,Y\right]  \right]
\right]  +\frac{1}{2}d_{1}d_{3}d_{4}\left[  Y,\left[  \left[  X,Y\right]
,X\right]  \right]  -\frac{1}{4}d_{1}d_{3}d_{4}\left[  \left[  \left[
X,Y\right]  ,X\right]  ,X\right] \\
& +\frac{1}{4}d_{1}d_{3}d_{4}\left[  \left[  \left[  X,Y\right]  ,Y\right]
,X\right]  -\frac{1}{4}d_{2}d_{3}d_{4}\left[  X,\left[  X,\left[  X,Y\right]
\right]  \right]  -\frac{1}{4}d_{2}d_{3}d_{4}\left[  X,\left[  Y,\left[
X,Y\right]  \right]  \right] \\
& +\frac{1}{2}d_{2}d_{3}d_{4}\left[  X,\left[  \left[  X,Y\right]  ,X\right]
\right]  -\frac{1}{4}d_{2}d_{3}d_{4}\left[  Y,\left[  X,\left[  X,Y\right]
\right]  \right]  -\frac{1}{4}d_{2}d_{3}d_{4}\left[  Y,\left[  Y,\left[
X,Y\right]  \right]  \right] \\
& +\frac{1}{2}d_{2}d_{3}d_{4}\left[  Y,\left[  \left[  X,Y\right]  ,X\right]
\right]  -\frac{1}{4}d_{2}d_{3}d_{4}\left[  \left[  \left[  X,Y\right]
,X\right]  ,X\right]  +\frac{1}{4}d_{2}d_{3}d_{4}\left[  \left[  \left[
X,Y\right]  ,Y\right]  ,X\right] \\
& -\frac{1}{5}d_{1}d_{2}d_{3}d_{4}\left[  X,\left[  X,\left[  X,\left[
X,Y\right]  \right]  \right]  \right]  -\frac{1}{5}d_{1}d_{2}d_{3}d_{4}\left[
X,\left[  X,\left[  Y,\left[  X,Y\right]  \right]  \right]  \right] \\
& +\frac{1}{2}d_{1}d_{2}d_{3}d_{4}\left[  X,\left[  X,\left[  \left[
X,Y\right]  ,X\right]  \right]  \right]  -\frac{1}{5}d_{1}d_{2}d_{3}%
d_{4}\left[  X,\left[  Y,\left[  X,\left[  X,Y\right]  \right]  \right]
\right] \\
& -\frac{1}{5}d_{1}d_{2}d_{3}d_{4}\left[  X,\left[  Y,\left[  Y,\left[
X,Y\right]  \right]  \right]  \right]  +\frac{1}{2}d_{1}d_{2}d_{3}d_{4}\left[
X,\left[  Y,\left[  \left[  X,Y\right]  ,X\right]  \right]  \right]
\end{align*}
\begin{align}
& -\frac{1}{3}d_{1}d_{2}d_{3}d_{4}\left[  X,\left[  \left[  \left[
X,Y\right]  ,X\right]  ,X\right]  \right]  +\frac{1}{3}d_{1}d_{2}d_{3}%
d_{4}\left[  X,\left[  \left[  \left[  X,Y\right]  ,Y\right]  ,X\right]
\right] \nonumber\\
& -\frac{1}{5}d_{1}d_{2}d_{3}d_{4}\left[  Y,\left[  X,\left[  X,\left[
X,Y\right]  \right]  \right]  \right]  -\frac{1}{5}d_{1}d_{2}d_{3}d_{4}\left[
Y,\left[  X,\left[  Y,\left[  X,Y\right]  \right]  \right]  \right]
\nonumber\\
& +\frac{1}{2}d_{1}d_{2}d_{3}d_{4}\left[  Y,\left[  X,\left[  \left[
X,Y\right]  ,X\right]  \right]  \right]  -\frac{1}{5}d_{1}d_{2}d_{3}%
d_{4}\left[  Y,\left[  Y,\left[  X,\left[  X,Y\right]  \right]  \right]
\right] \nonumber\\
& -\frac{1}{5}d_{1}d_{2}d_{3}d_{4}\left[  Y,\left[  Y,\left[  Y,\left[
X,Y\right]  \right]  \right]  \right]  +\frac{1}{2}d_{1}d_{2}d_{3}d_{4}\left[
Y,\left[  Y,\left[  \left[  X,Y\right]  ,X\right]  \right]  \right]
\nonumber\\
& -\frac{1}{3}d_{1}d_{2}d_{3}d_{4}\left[  Y,\left[  \left[  \left[
X,Y\right]  ,X\right]  ,X\right]  \right]  +\frac{1}{3}d_{1}d_{2}d_{3}%
d_{4}\left[  Y,\left[  \left[  \left[  X,Y\right]  ,Y\right]  ,X\right]
\right] \nonumber\\
& -\frac{1}{2}d_{1}d_{2}d_{3}d_{4}\left[  \left[  X,Y\right]  ,\left[
X,\left[  X,Y\right]  \right]  \right]  -\frac{1}{2}d_{1}d_{2}d_{3}%
d_{4}\left[  \left[  X,Y\right]  ,\left[  Y,\left[  X,Y\right]  \right]
\right] \nonumber\\
& +d_{1}d_{2}d_{3}d_{4}\left[  \left[  X,Y\right]  ,\left[  \left[
X,Y\right]  ,X\right]  \right]  +\frac{1}{8}d_{1}d_{2}d_{3}d_{4}\left[
\left[  X,\left[  \left[  X,Y\right]  ,X\right]  \right]  ,X\right]
\nonumber\\
& +\frac{1}{8}d_{1}d_{2}d_{3}d_{4}\left[  \left[  X,\left[  \left[
X,Y\right]  ,Y\right]  \right]  ,X\right]  +\frac{1}{8}d_{1}d_{2}d_{3}%
d_{4}\left[  \left[  Y,\left[  \left[  X,Y\right]  ,X\right]  \right]
,X\right] \nonumber\\
& +\frac{1}{8}d_{1}d_{2}d_{3}d_{4}\left[  \left[  Y,\left[  \left[
X,Y\right]  ,Y\right]  \right]  ,X\right]  +\frac{1}{3}d_{1}d_{2}d_{3}%
d_{4}\left[  \left[  Y,\left[  \left[  X,Y\right]  ,X\right]  \right]
,\left[  X,Y\right]  \right] \nonumber\\
& -\frac{1}{3}d_{1}d_{2}d_{3}d_{4}\left[  \left[  \left[  X,Y\right]
,Y\right]  ,\left[  X,Y\right]  \right]  +\frac{1}{8}d_{1}d_{2}d_{3}%
d_{4}\left[  \left[  \left[  \left[  X,Y\right]  ,X\right]  ,X\right]
,X\right] \nonumber\\
& -\frac{1}{8}d_{1}d_{2}d_{3}d_{4}\left[  \left[  \left[  \left[  X,Y\right]
,X\right]  ,Y\right]  ,X\right]  -\frac{1}{8}d_{1}d_{2}d_{3}d_{4}\left[
\left[  \left[  \left[  X,Y\right]  ,Y\right]  ,X\right]  ,X\right]
\nonumber\\
& +\frac{1}{8}d_{1}d_{2}d_{3}d_{4}\left[  \left[  \left[  \left[  X,Y\right]
,Y\right]  ,Y\right]  ,X\right] \label{4.10}%
\end{align}

Letting $m_{51}$ be the right-hand side of (\ref{4.10}) with the first term
$X$ deleted, we have
\begin{align}
& \left(  \ref{4.9}\right) \nonumber\\
& =\exp\;-d_{5}m_{51}.\exp\;d_{5}X+d_{5}m_{51}.\nonumber\\
& \exp\;\left(
\begin{array}
[c]{c}%
\left(  d_{1}+d_{2}+d_{3}+d_{4}\right)  \left(  X+Y\right)  +d_{5}Y+\frac
{1}{2}\left(  d_{1}+d_{2}+d_{3}+d_{4}\right)  ^{2}\left[  X,Y\right] \\
+\frac{1}{12}\left(  d_{1}+d_{2}+d_{3}+d_{4}\right)  ^{3}\left[  \left[
X,Y\right]  ,Y-X\right] \\
+\frac{1}{96}\left(  d_{1}+d_{2}+d_{3}+d_{4}\right)  ^{4}\left(  \left[
X+Y,\left[  \left[  X,Y\right]  ,X+Y\right]  \right]  +\left[  \left[  \left[
X,Y\right]  ,X-Y\right]  ,X-Y\right]  \right) \\
+d_{5}i_{51}+d_{5}i_{52}+d_{5}i_{53}%
\end{array}
\right)  .\nonumber\\
& \exp\;-d_{5}n_{51}-d_{5}n_{52}-d_{5}n_{53}-d_{5}n_{54}\nonumber\\
& \left)  \text{By Proposition \ref{t2.2}}\right( \nonumber\\
& =\exp\;-d_{5}m_{51}.\left(  X+m_{51}\right)  _{d_{5}}.\nonumber\\
& \exp\;\left(
\begin{array}
[c]{c}%
\left(  d_{1}+d_{2}+d_{3}+d_{4}\right)  \left(  X+Y\right)  +d_{5}Y+\frac
{1}{2}\left(  d_{1}+d_{2}+d_{3}+d_{4}\right)  ^{2}\left[  X,Y\right] \\
+\frac{1}{12}\left(  d_{1}+d_{2}+d_{3}+d_{4}\right)  ^{3}\left[  \left[
X,Y\right]  ,Y-X\right] \\
+\frac{1}{96}\left(  d_{1}+d_{2}+d_{3}+d_{4}\right)  ^{4}\left(  \left[
X+Y,\left[  \left[  X,Y\right]  ,X+Y\right]  \right]  +\left[  \left[  \left[
X,Y\right]  ,X-Y\right]  ,X-Y\right]  \right) \\
+d_{5}i_{51}+d_{5}i_{52}+d_{5}i_{53}%
\end{array}
\right)  .\nonumber\\
& \exp\;-d_{5}n_{51}-d_{5}n_{52}-d_{5}n_{53}-d_{5}n_{54}\nonumber\\
& \left)  \text{By Proposition \ref{t2.3}}\right( \nonumber\\
& =\exp\;-d_{5}m_{51}.\nonumber\\
& \exp\;\left(
\begin{array}
[c]{c}%
\left(  d_{1}+d_{2}+d_{3}+d_{4}+d_{5}\right)  \left(  X+Y\right)  +\frac{1}%
{2}\left(  d_{1}+d_{2}+d_{3}+d_{4}\right)  ^{2}\left[  X,Y\right] \\
+\frac{1}{12}\left(  d_{1}+d_{2}+d_{3}+d_{4}\right)  ^{3}\left[  \left[
X,Y\right]  ,Y-X\right] \\
+\frac{1}{96}\left(  d_{1}+d_{2}+d_{3}+d_{4}\right)  ^{4}\left(  \left[
X+Y,\left[  \left[  X,Y\right]  ,X+Y\right]  \right]  +\left[  \left[  \left[
X,Y\right]  ,X-Y\right]  ,X-Y\right]  \right) \\
+d_{5}i_{51}+d_{5}i_{52}+d_{5}i_{53}%
\end{array}
\right)  .\nonumber\\
& \exp\;-d_{5}n_{51}-d_{5}n_{52}-d_{5}n_{53}-d_{5}n_{54}\nonumber\\
& \left)  \text{By (\ref{4.10})}\right( \label{4.11}%
\end{align}

We let $j_{51}$ be the result of $-m_{51}$ by deleting all the terms whose
coefficients contain $d_{1}d_{2}d_{3}d_{4}$. Then, due to Theorem \ref{t2.1},
we have
\begin{align*}
& \delta^{\mathrm{right}}\left(  \exp\right)  \left(
\begin{array}
[c]{c}%
\left(  d_{1}+d_{2}+d_{3}+d_{4}+d_{5}\right)  \left(  X+Y\right)  +\frac{1}%
{2}\left(  d_{1}+d_{2}+d_{3}+d_{4}\right)  ^{2}\left[  X,Y\right] \\
+\frac{1}{12}\left(  d_{1}+d_{2}+d_{3}+d_{4}\right)  ^{3}\left[  \left[
X,Y\right]  ,Y-X\right] \\
+\frac{1}{96}\left(  d_{1}+d_{2}+d_{3}+d_{4}\right)  ^{4}\left[  X+Y,\left[
\left[  X,Y\right]  ,X+Y\right]  \right] \\
+\frac{1}{96}\left(  d_{1}+d_{2}+d_{3}+d_{4}\right)  ^{4}\left[  \left[
\left[  X,Y\right]  ,X-Y\right]  ,X-Y\right] \\
+d_{5}i_{51}+d_{5}i_{52}+d_{5}i_{53}%
\end{array}
\right)  \left(  j_{51}\right) \\
& =\frac{1}{2}d_{1}\left[  X,Y\right]  +\frac{1}{2}d_{2}\left[  X,Y\right]
+\frac{5}{6}d_{1}d_{2}\left[  X,\left[  X,Y\right]  \right]  +\frac{5}{6}%
d_{1}d_{2}\left[  Y,\left[  X,Y\right]  \right]  -\frac{1}{2}d_{1}d_{2}\left[
\left[  X,Y\right]  ,X\right] \\
& +\frac{1}{2}d_{3}\left[  X,Y\right]  +\frac{5}{6}d_{1}d_{3}\left[  X,\left[
X,Y\right]  \right]  +\frac{5}{6}d_{1}d_{3}\left[  Y,\left[  X,Y\right]
\right]  -\frac{1}{2}d_{1}d_{3}\left[  \left[  X,Y\right]  ,X\right] \\
& +\frac{5}{6}d_{2}d_{3}\left[  X,\left[  X,Y\right]  \right]  +\frac{5}%
{6}d_{2}d_{3}\left[  Y,\left[  X,Y\right]  \right]  -\frac{1}{2}d_{2}%
d_{3}\left[  \left[  X,Y\right]  ,X\right]  +\frac{5}{4}d_{1}d_{2}d_{3}\left[
X,\left[  X,\left[  X,Y\right]  \right]  \right] \\
& +\frac{5}{4}d_{1}d_{2}d_{3}\left[  X,\left[  Y,\left[  X,Y\right]  \right]
\right]  -\frac{5}{4}d_{1}d_{2}d_{3}\left[  X,\left[  \left[  X,Y\right]
,X\right]  \right]  +\frac{5}{4}d_{1}d_{2}d_{3}\left[  Y,\left[  X,\left[
X,Y\right]  \right]  \right] \\
& +\frac{5}{4}d_{1}d_{2}d_{3}\left[  Y,\left[  Y,\left[  X,Y\right]  \right]
\right]  -\frac{5}{4}d_{1}d_{2}d_{3}\left[  Y,\left[  \left[  X,Y\right]
,X\right]  \right]  +\frac{1}{4}d_{1}d_{2}d_{3}\left[  \left[  \left[
X,Y\right]  ,X\right]  ,X\right] \\
& -\frac{1}{4}d_{1}d_{2}d_{3}\left[  \left[  \left[  X,Y\right]  ,Y\right]
,X\right]  +\frac{1}{2}d_{4}\left[  X,Y\right]  +\frac{5}{6}d_{1}d_{4}\left[
X,\left[  X,Y\right]  \right]  +\frac{5}{6}d_{1}d_{4}\left[  Y,\left[
X,Y\right]  \right] \\
& -\frac{1}{2}d_{1}d_{4}\left[  \left[  X,Y\right]  ,X\right]  +\frac{5}%
{6}d_{2}d_{4}\left[  X,\left[  X,Y\right]  \right]  +\frac{5}{6}d_{2}%
d_{4}\left[  Y,\left[  X,Y\right]  \right]  -\frac{1}{2}d_{2}d_{4}\left[
\left[  X,Y\right]  ,X\right] \\
& +\frac{5}{4}d_{1}d_{2}d_{4}\left[  X,\left[  X,\left[  X,Y\right]  \right]
\right]  +\frac{5}{4}d_{1}d_{2}d_{4}\left[  X,\left[  Y,\left[  X,Y\right]
\right]  \right]  -\frac{5}{4}d_{1}d_{2}d_{4}\left[  X,\left[  \left[
X,Y\right]  ,X\right]  \right] \\
& +\frac{5}{4}d_{1}d_{2}d_{4}\left[  Y,\left[  X,\left[  X,Y\right]  \right]
\right]  +\frac{5}{4}d_{1}d_{2}d_{4}\left[  Y,\left[  Y,\left[  X,Y\right]
\right]  \right]  -\frac{5}{4}d_{1}d_{2}d_{4}\left[  Y,\left[  \left[
X,Y\right]  ,X\right]  \right] \\
& +\frac{1}{4}d_{1}d_{2}d_{4}\left[  \left[  \left[  X,Y\right]  ,X\right]
,X\right]  -\frac{1}{4}d_{1}d_{2}d_{4}\left[  \left[  \left[  X,Y\right]
,Y\right]  ,X\right]  +\frac{5}{6}d_{3}d_{4}\left[  X,\left[  X,Y\right]
\right] \\
& +\frac{5}{6}d_{3}d_{4}\left[  Y,\left[  X,Y\right]  \right]  -\frac{1}%
{2}d_{3}d_{4}\left[  \left[  X,Y\right]  ,X\right]  +\frac{5}{4}d_{1}%
d_{3}d_{4}\left[  X,\left[  X,\left[  X,Y\right]  \right]  \right] \\
& +\frac{5}{4}d_{1}d_{3}d_{4}\left[  X,\left[  Y,\left[  X,Y\right]  \right]
\right]  -\frac{5}{4}d_{1}d_{3}d_{4}\left[  X,\left[  \left[  X,Y\right]
,X\right]  \right]  +\frac{5}{4}d_{1}d_{3}d_{4}\left[  Y,\left[  X,\left[
X,Y\right]  \right]  \right] \\
& +\frac{5}{4}d_{1}d_{3}d_{4}\left[  Y,\left[  Y,\left[  X,Y\right]  \right]
\right]  -\frac{5}{4}d_{1}d_{3}d_{4}\left[  Y,\left[  \left[  X,Y\right]
,X\right]  \right]  +\frac{1}{4}d_{1}d_{3}d_{4}\left[  \left[  \left[
X,Y\right]  ,X\right]  ,X\right] \\
& -\frac{1}{4}d_{1}d_{3}d_{4}\left[  \left[  \left[  X,Y\right]  ,Y\right]
,X\right]  +\frac{5}{4}d_{1}d_{3}d_{4}\left[  X,\left[  X,\left[  X,Y\right]
\right]  \right]  +\frac{5}{4}d_{1}d_{3}d_{4}\left[  X,\left[  Y,\left[
X,Y\right]  \right]  \right] \\
& -\frac{5}{4}d_{1}d_{3}d_{4}\left[  X,\left[  \left[  X,Y\right]  ,X\right]
\right]  +\frac{5}{4}d_{2}d_{3}d_{4}\left[  Y,\left[  X,\left[  X,Y\right]
\right]  \right]  +\frac{5}{4}d_{2}d_{3}d_{4}\left[  Y,\left[  Y,\left[
X,Y\right]  \right]  \right] \\
& -\frac{5}{4}d_{2}d_{3}d_{4}\left[  Y,\left[  \left[  X,Y\right]  ,X\right]
\right]  +\frac{1}{4}d_{2}d_{3}d_{4}\left[  \left[  \left[  X,Y\right]
,X\right]  ,X\right]  -\frac{1}{4}d_{2}d_{3}d_{4}\left[  \left[  \left[
X,Y\right]  ,Y\right]  ,X\right] \\
& +\frac{5}{3}d_{1}d_{2}d_{3}d_{4}\left[  X,\left[  X,\left[  X,\left[
X,Y\right]  \right]  \right]  \right]  +\frac{5}{3}d_{1}d_{2}d_{3}d_{4}\left[
X,\left[  X,\left[  Y,\left[  X,Y\right]  \right]  \right]  \right] \\
& -2d_{1}d_{2}d_{3}d_{4}\left[  X,\left[  X,\left[  \left[  X,Y\right]
,X\right]  \right]  \right]  +\frac{5}{3}d_{1}d_{2}d_{3}d_{4}\left[  X,\left[
Y,\left[  X,\left[  X,Y\right]  \right]  \right]  \right] \\
& +\frac{5}{3}d_{1}d_{2}d_{3}d_{4}\left[  X,\left[  Y,\left[  Y,\left[
X,Y\right]  \right]  \right]  \right]  -2d_{1}d_{2}d_{3}d_{4}\left[  X,\left[
Y,\left[  \left[  X,Y\right]  ,X\right]  \right]  \right]
\end{align*}
\begin{align}
& +\frac{1}{2}d_{1}d_{2}d_{3}d_{4}\left[  X,\left[  \left[  \left[
X,Y\right]  ,X\right]  ,X\right]  \right]  -\frac{1}{2}d_{1}d_{2}d_{3}%
d_{4}\left[  X,\left[  \left[  \left[  X,Y\right]  ,Y\right]  ,X\right]
\right] \nonumber\\
& +\frac{5}{3}d_{1}d_{2}d_{3}d_{4}\left[  Y,\left[  X,\left[  X,\left[
X,Y\right]  \right]  \right]  \right]  +\frac{5}{3}d_{1}d_{2}d_{3}d_{4}\left[
Y,\left[  X,\left[  Y,\left[  X,Y\right]  \right]  \right]  \right]
\nonumber\\
& -2d_{1}d_{2}d_{3}d_{4}\left[  Y,\left[  X,\left[  \left[  X,Y\right]
,X\right]  \right]  \right]  +\frac{5}{3}d_{1}d_{2}d_{3}d_{4}\left[  Y,\left[
Y,\left[  X,\left[  X,Y\right]  \right]  \right]  \right] \nonumber\\
& +\frac{5}{3}d_{1}d_{2}d_{3}d_{4}\left[  Y,\left[  Y,\left[  Y,\left[
X,Y\right]  \right]  \right]  \right]  -2d_{1}d_{2}d_{3}d_{4}\left[  Y,\left[
Y,\left[  \left[  X,Y\right]  ,X\right]  \right]  \right] \nonumber\\
& +\frac{1}{2}d_{1}d_{2}d_{3}d_{4}\left[  Y,\left[  \left[  \left[
X,Y\right]  ,X\right]  ,X\right]  \right]  -\frac{1}{2}d_{1}d_{2}d_{3}%
d_{4}\left[  Y,\left[  \left[  \left[  X,Y\right]  ,Y\right]  ,X\right]
\right] \nonumber\\
& +2d_{1}d_{2}d_{3}d_{4}\left[  \left[  X,Y\right]  ,\left[  X,\left[
X,Y\right]  \right]  \right]  +2d_{1}d_{2}d_{3}d_{4}\left[  \left[
X,Y\right]  ,\left[  Y,\left[  X,Y\right]  \right]  \right] \nonumber\\
& -\frac{3}{2}d_{1}d_{2}d_{3}d_{4}\left[  \left[  X,Y\right]  ,\left[  \left[
X,Y\right]  ,X\right]  \right]  -\frac{1}{2}d_{1}d_{2}d_{3}d_{4}\left[
\left[  \left[  X,Y\right]  ,X\right]  ,\left[  X,Y\right]  \right]
\nonumber\\
& +\frac{1}{2}d_{1}d_{2}d_{3}d_{4}\left[  \left[  \left[  X,Y\right]
,Y\right]  ,\left[  X,Y\right]  \right] \label{4.12}%
\end{align}

Letting $m_{52}$ be the right-hand side of (\ref{4.12}), we have
\begin{align}
& \left(  \ref{4.11}\right) \nonumber\\
& =\exp\;-d_{5}m_{51}-d_{5}m_{52}.\exp\;d_{5}m_{52}\nonumber\\
& \exp\;\left(
\begin{array}
[c]{c}%
\left(  d_{1}+d_{2}+d_{3}+d_{4}+d_{5}\right)  \left(  X+Y\right)  +\frac{1}%
{2}\left(  d_{1}+d_{2}+d_{3}+d_{4}\right)  ^{2}\left[  X,Y\right] \\
+\frac{1}{12}\left(  d_{1}+d_{2}+d_{3}+d_{4}\right)  ^{3}\left[  \left[
X,Y\right]  ,Y-X\right] \\
+\frac{1}{96}\left(  d_{1}+d_{2}+d_{3}+d_{4}\right)  ^{4}\left(  \left[
X+Y,\left[  \left[  X,Y\right]  ,X+Y\right]  \right]  +\left[  \left[  \left[
X,Y\right]  ,X-Y\right]  ,X-Y\right]  \right) \\
+d_{5}i_{51}+d_{5}i_{52}+d_{5}i_{53}%
\end{array}
\right)  .\nonumber\\
& \exp\;-d_{5}n_{51}-d_{5}n_{52}-d_{5}n_{53}-d_{5}n_{54}\nonumber\\
& \left)  \text{By Proposition \ref{t2.2}}\right( \nonumber\\
& =\exp\;-d_{5}m_{51}-d_{5}m_{52}.\left(  m_{52}\right)  _{d_{5}}.\nonumber\\
& \exp\;\left(
\begin{array}
[c]{c}%
\left(  d_{1}+d_{2}+d_{3}+d_{4}+d_{5}\right)  \left(  X+Y\right)  +\frac{1}%
{2}\left(  d_{1}+d_{2}+d_{3}+d_{4}\right)  ^{2}\left[  X,Y\right] \\
+\frac{1}{12}\left(  d_{1}+d_{2}+d_{3}+d_{4}\right)  ^{3}\left[  \left[
X,Y\right]  ,Y-X\right] \\
+\frac{1}{96}\left(  d_{1}+d_{2}+d_{3}+d_{4}\right)  ^{4}\left(  \left[
X+Y,\left[  \left[  X,Y\right]  ,X+Y\right]  \right]  +\left[  \left[  \left[
X,Y\right]  ,X-Y\right]  ,X-Y\right]  \right) \\
+d_{5}i_{51}+d_{5}i_{52}+d_{5}i_{53}%
\end{array}
\right)  .\nonumber\\
& \exp\;-d_{5}n_{51}-d_{5}n_{52}-d_{5}n_{53}-d_{5}n_{54}\nonumber\\
& \left)  \text{By Proposition \ref{t2.3}}\right( \nonumber\\
& =\exp\;-d_{5}m_{51}-d_{5}m_{52}.\nonumber\\
& \exp\;\left(
\begin{array}
[c]{c}%
\left(  d_{1}+d_{2}+d_{3}+d_{4}+d_{5}\right)  \left(  X+Y\right)  +\frac{1}%
{2}\left(  d_{1}+d_{2}+d_{3}+d_{4}\right)  ^{2}\left[  X,Y\right] \\
+\frac{1}{12}\left(  d_{1}+d_{2}+d_{3}+d_{4}\right)  ^{3}\left[  \left[
X,Y\right]  ,Y-X\right] \\
+\frac{1}{96}\left(  d_{1}+d_{2}+d_{3}+d_{4}\right)  ^{4}\left(  \left[
X+Y,\left[  \left[  X,Y\right]  ,X+Y\right]  \right]  +\left[  \left[  \left[
X,Y\right]  ,X-Y\right]  ,X-Y\right]  \right) \\
+d_{5}i_{51}+d_{5}i_{52}+d_{5}i_{53}+d_{5}j_{51}%
\end{array}
\right)  .\nonumber\\
& \exp\;-d_{5}n_{51}-d_{5}n_{52}-d_{5}n_{53}-d_{5}n_{54}\nonumber\\
& \left)  \text{By (\ref{4.12})}\right( \label{4.13}%
\end{align}

We let $j_{52}$ be the result of $-m_{51}-m_{52}$ by deleting all the terms
whose coefficients contain $d_{1}d_{2}d_{3}d_{4}$. Then, by dint of Theorem
\ref{t2.1}, we have
\begin{align}
& \delta^{\mathrm{right}}\left(  \exp\right)  \left(
\begin{array}
[c]{c}%
\left(  d_{1}+d_{2}+d_{3}+d_{4}+d_{5}\right)  \left(  X+Y\right)  +\frac{1}%
{2}\left(  d_{1}+d_{2}+d_{3}+d_{4}\right)  ^{2}\left[  X,Y\right] \\
+\frac{1}{12}\left(  d_{1}+d_{2}+d_{3}+d_{4}\right)  ^{3}\left[  \left[
X,Y\right]  ,Y-X\right] \\
+\frac{1}{96}\left(  d_{1}+d_{2}+d_{3}+d_{4}\right)  ^{4}\left[  X+Y,\left[
\left[  X,Y\right]  ,X+Y\right]  \right] \\
+\frac{1}{96}\left(  d_{1}+d_{2}+d_{3}+d_{4}\right)  ^{4}\left[  \left[
\left[  X,Y\right]  ,X-Y\right]  ,X-Y\right] \\
+d_{5}i_{51}+d_{5}i_{52}+d_{5}i_{53}+d_{5}j_{51}%
\end{array}
\right)  \left(  j_{52}\right) \nonumber\\
& =-\frac{1}{2}d_{1}d_{2}\left[  X,\left[  X,Y\right]  \right]  -\frac{1}%
{2}d_{1}d_{2}\left[  Y,\left[  X,Y\right]  \right]  -\frac{1}{2}d_{1}%
d_{3}\left[  X,\left[  X,Y\right]  \right]  -\frac{1}{2}d_{1}d_{3}\left[
Y,\left[  X,Y\right]  \right] \nonumber\\
& -\frac{1}{2}d_{2}d_{3}\left[  X,\left[  X,Y\right]  \right]  -\frac{1}%
{2}d_{2}d_{3}\left[  Y,\left[  X,Y\right]  \right]  -\frac{7}{4}d_{1}%
d_{2}d_{3}\left[  X,\left[  X,\left[  X,Y\right]  \right]  \right] \nonumber\\
& -\frac{7}{4}d_{1}d_{2}d_{3}\left[  X,\left[  Y,\left[  X,Y\right]  \right]
\right]  +\frac{3}{4}d_{1}d_{2}d_{3}\left[  X,\left[  \left[  X,Y\right]
,X\right]  \right]  -\frac{7}{4}d_{1}d_{2}d_{3}\left[  Y,\left[  X,\left[
X,Y\right]  \right]  \right] \nonumber\\
& -\frac{7}{4}d_{1}d_{2}d_{3}\left[  Y,\left[  Y,\left[  X,Y\right]  \right]
\right]  +\frac{3}{4}d_{1}d_{2}d_{3}\left[  Y,\left[  \left[  X,Y\right]
,X\right]  \right]  -\frac{1}{2}d_{1}d_{4}\left[  X,\left[  X,Y\right]
\right] \nonumber\\
& -\frac{1}{2}d_{1}d_{4}\left[  Y,\left[  X,Y\right]  \right]  -\frac{1}%
{2}d_{2}d_{4}\left[  X,\left[  X,Y\right]  \right]  -\frac{1}{2}d_{2}%
d_{4}\left[  Y,\left[  X,Y\right]  \right]  -\frac{7}{4}d_{1}d_{2}d_{4}\left[
X,\left[  X,\left[  X,Y\right]  \right]  \right] \nonumber\\
& -\frac{7}{4}d_{1}d_{2}d_{4}\left[  X,\left[  Y,\left[  X,Y\right]  \right]
\right]  +\frac{3}{4}d_{1}d_{2}d_{4}\left[  X,\left[  \left[  X,Y\right]
,X\right]  \right]  -\frac{7}{4}d_{1}d_{2}d_{4}\left[  Y,\left[  X,\left[
X,Y\right]  \right]  \right] \nonumber\\
& -\frac{7}{4}d_{1}d_{2}d_{4}\left[  Y,\left[  Y,\left[  X,Y\right]  \right]
\right]  +\frac{3}{4}d_{1}d_{2}d_{4}\left[  Y,\left[  \left[  X,Y\right]
,X\right]  \right]  -\frac{1}{2}d_{3}d_{4}\left[  X,\left[  X,Y\right]
\right] \nonumber\\
& -\frac{1}{2}d_{3}d_{4}\left[  Y,\left[  X,Y\right]  \right]  -\frac{7}%
{4}d_{1}d_{3}d_{4}\left[  X,\left[  X,\left[  X,Y\right]  \right]  \right]
-\frac{7}{4}d_{1}d_{3}d_{4}\left[  X,\left[  Y,\left[  X,Y\right]  \right]
\right] \nonumber\\
& +\frac{3}{4}d_{1}d_{3}d_{4}\left[  X,\left[  \left[  X,Y\right]  ,X\right]
\right]  -\frac{7}{4}d_{1}d_{3}d_{4}\left[  Y,\left[  X,\left[  X,Y\right]
\right]  \right]  -\frac{7}{4}d_{1}d_{3}d_{4}\left[  Y,\left[  Y,\left[
X,Y\right]  \right]  \right] \nonumber\\
& +\frac{3}{4}d_{1}d_{3}d_{4}\left[  Y,\left[  \left[  X,Y\right]  ,X\right]
\right]  -\frac{7}{4}d_{2}d_{3}d_{4}\left[  X,\left[  X,\left[  X,Y\right]
\right]  \right]  -\frac{7}{4}d_{2}d_{3}d_{4}\left[  X,\left[  Y,\left[
X,Y\right]  \right]  \right] \nonumber\\
& +\frac{3}{4}d_{2}d_{3}d_{4}\left[  X,\left[  \left[  X,Y\right]  ,X\right]
\right]  -\frac{7}{4}d_{2}d_{3}d_{4}\left[  Y,\left[  X,\left[  X,Y\right]
\right]  \right]  -\frac{7}{4}d_{2}d_{3}d_{4}\left[  Y,\left[  Y,\left[
X,Y\right]  \right]  \right] \nonumber\\
& +\frac{3}{4}d_{2}d_{3}d_{4}\left[  Y,\left[  \left[  X,Y\right]  ,X\right]
\right]  -3d_{1}d_{2}d_{3}d_{4}\left[  X,\left[  X,\left[  X,\left[
X,Y\right]  \right]  \right]  \right] \nonumber\\
& -3d_{1}d_{2}d_{3}d_{4}\left[  X,\left[  X,\left[  Y,\left[  X,Y\right]
\right]  \right]  \right]  +\frac{1}{2}d_{1}d_{2}d_{3}d_{4}\left[  X,\left[
X,\left[  \left[  X,Y\right]  ,X\right]  \right]  \right] \nonumber\\
& -3d_{1}d_{2}d_{3}d_{4}\left[  X,\left[  Y,\left[  X,\left[  X,Y\right]
\right]  \right]  \right]  -3d_{1}d_{2}d_{3}d_{4}\left[  X,\left[  Y,\left[
Y,\left[  X,Y\right]  \right]  \right]  \right] \nonumber\\
& +\frac{1}{2}d_{1}d_{2}d_{3}d_{4}\left[  X,\left[  Y,\left[  \left[
X,Y\right]  ,X\right]  \right]  \right]  -3d_{1}d_{2}d_{3}d_{4}\left[
Y,\left[  X,\left[  X,\left[  X,Y\right]  \right]  \right]  \right]
\nonumber\\
& -3d_{1}d_{2}d_{3}d_{4}\left[  Y,\left[  X,\left[  Y,\left[  X,Y\right]
\right]  \right]  \right]  -3d_{1}d_{2}d_{3}d_{4}\left[  Y,\left[  Y,\left[
X,\left[  X,Y\right]  \right]  \right]  \right] \nonumber\\
& -3d_{1}d_{2}d_{3}d_{4}\left[  Y,\left[  Y,\left[  Y,\left[  X,Y\right]
\right]  \right]  \right]  +\frac{1}{2}d_{1}d_{2}d_{3}d_{4}\left[  Y,\left[
X,\left[  \left[  X,Y\right]  ,X\right]  \right]  \right] \nonumber\\
& +\frac{1}{2}d_{1}d_{2}d_{3}d_{4}\left[  Y,\left[  Y,\left[  \left[
X,Y\right]  ,X\right]  \right]  \right]  -\frac{3}{2}d_{1}d_{2}d_{3}%
d_{4}\left[  \left[  X,Y\right]  ,\left[  X,\left[  X,Y\right]  \right]
\right] \nonumber\\
& -\frac{3}{2}d_{1}d_{2}d_{3}d_{4}\left[  \left[  X,Y\right]  ,\left[
Y,\left[  X,Y\right]  \right]  \right] \label{4.14}%
\end{align}

Letting $m_{53}$ be the right-hand side of (\ref{4.14}), we have
\begin{align}
& \left(  \ref{4.13}\right) \nonumber\\
& =\exp\;-d_{5}m_{51}-d_{5}m_{52}-d_{5}m_{53}.\exp\;d_{5}m_{53}\nonumber\\
& \exp\;\left(
\begin{array}
[c]{c}%
\left(  d_{1}+d_{2}+d_{3}+d_{4}+d_{5}\right)  \left(  X+Y\right)  +\frac{1}%
{2}\left(  d_{1}+d_{2}+d_{3}+d_{4}\right)  ^{2}\left[  X,Y\right] \\
+\frac{1}{12}\left(  d_{1}+d_{2}+d_{3}+d_{4}\right)  ^{3}\left[  \left[
X,Y\right]  ,Y-X\right] \\
+\frac{1}{96}\left(  d_{1}+d_{2}+d_{3}+d_{4}\right)  ^{4}\left(  \left[
X+Y,\left[  \left[  X,Y\right]  ,X+Y\right]  \right]  +\left[  \left[  \left[
X,Y\right]  ,X-Y\right]  ,X-Y\right]  \right) \\
+d_{5}i_{51}+d_{5}i_{52}+d_{5}i_{53}+d_{5}j_{51}%
\end{array}
\right)  .\nonumber\\
& \exp\;-d_{5}n_{51}-d_{5}n_{52}-d_{5}n_{53}-d_{5}n_{54}\nonumber\\
& \left)  \text{By Proposition \ref{t2.2}}\right( \nonumber\\
& =\exp\;-d_{5}m_{51}-d_{5}m_{52}-d_{5}m_{53}.\left(  m_{53}\right)  _{d_{5}%
}.\nonumber\\
& \exp\;\left(
\begin{array}
[c]{c}%
\left(  d_{1}+d_{2}+d_{3}+d_{4}+d_{5}\right)  \left(  X+Y\right)  +\frac{1}%
{2}\left(  d_{1}+d_{2}+d_{3}+d_{4}\right)  ^{2}\left[  X,Y\right] \\
+\frac{1}{12}\left(  d_{1}+d_{2}+d_{3}+d_{4}\right)  ^{3}\left[  \left[
X,Y\right]  ,Y-X\right] \\
+\frac{1}{96}\left(  d_{1}+d_{2}+d_{3}+d_{4}\right)  ^{4}\left(  \left[
X+Y,\left[  \left[  X,Y\right]  ,X+Y\right]  \right]  +\left[  \left[  \left[
X,Y\right]  ,X-Y\right]  ,X-Y\right]  \right) \\
+d_{5}i_{51}+d_{5}i_{52}+d_{5}i_{53}+d_{5}j_{51}%
\end{array}
\right)  .\nonumber\\
& \exp\;-d_{5}n_{51}-d_{5}n_{52}-d_{5}n_{53}-d_{5}n_{54}\nonumber\\
& \left)  \text{By Proposition \ref{t2.3}}\right( \nonumber\\
& =\exp\;-d_{5}m_{51}-d_{5}m_{52}-d_{5}m_{53}.\nonumber\\
& \exp\;\left(
\begin{array}
[c]{c}%
\left(  d_{1}+d_{2}+d_{3}+d_{4}+d_{5}\right)  \left(  X+Y\right)  +\frac{1}%
{2}\left(  d_{1}+d_{2}+d_{3}+d_{4}\right)  ^{2}\left[  X,Y\right] \\
+\frac{1}{12}\left(  d_{1}+d_{2}+d_{3}+d_{4}\right)  ^{3}\left[  \left[
X,Y\right]  ,Y-X\right] \\
+\frac{1}{96}\left(  d_{1}+d_{2}+d_{3}+d_{4}\right)  ^{4}\left(  \left[
X+Y,\left[  \left[  X,Y\right]  ,X+Y\right]  \right]  +\left[  \left[  \left[
X,Y\right]  ,X-Y\right]  ,X-Y\right]  \right) \\
+d_{5}i_{51}+d_{5}i_{52}+d_{5}i_{53}+d_{5}j_{51}+d_{5}j_{52}%
\end{array}
\right)  .\nonumber\\
& \exp\;-d_{5}n_{51}-d_{5}n_{52}-d_{5}n_{53}-d_{5}n_{54}\nonumber\\
& \left)  \text{By (\ref{4.14})}\right( \label{4.15}%
\end{align}

We let $j_{53}$ be the result of $-m_{51}-m_{52}-m_{53}$ by deleting all the
terms whose coefficients contain $d_{1}d_{2}d_{3}d_{4}$. Then, thanks to
Theorem \ref{t2.1}, we have
\begin{align}
& \delta^{\mathrm{right}}\left(  \exp\right)  \left(
\begin{array}
[c]{c}%
\left(  d_{1}+d_{2}+d_{3}+d_{4}+d_{5}\right)  \left(  X+Y\right)  +\frac{1}%
{2}\left(  d_{1}+d_{2}+d_{3}+d_{4}\right)  ^{2}\left[  X,Y\right] \\
+\frac{1}{12}\left(  d_{1}+d_{2}+d_{3}+d_{4}\right)  ^{3}\left[  \left[
X,Y\right]  ,Y-X\right] \\
+\frac{1}{96}\left(  d_{1}+d_{2}+d_{3}+d_{4}\right)  ^{4}\left[  X+Y,\left[
\left[  X,Y\right]  ,X+Y\right]  \right] \\
+\frac{1}{96}\left(  d_{1}+d_{2}+d_{3}+d_{4}\right)  ^{4}\left[  \left[
\left[  X,Y\right]  ,X-Y\right]  ,X-Y\right] \\
+d_{5}i_{51}+d_{5}i_{52}+d_{5}i_{53}+d_{5}j_{51}+d_{5}j_{52}%
\end{array}
\right)  \left(  j_{53}\right) \nonumber\\
& =\frac{3}{4}d_{1}d_{2}d_{3}\left[  X,\left[  X,\left[  X,Y\right]  \right]
\right]  +\frac{3}{4}d_{1}d_{2}d_{3}\left[  X,\left[  Y,\left[  X,Y\right]
\right]  \right]  +\frac{3}{4}d_{1}d_{2}d_{3}\left[  Y,\left[  X,\left[
X,Y\right]  \right]  \right] \nonumber\\
& +\frac{3}{4}d_{1}d_{2}d_{3}\left[  Y,\left[  Y,\left[  X,Y\right]  \right]
\right]  +\frac{3}{4}d_{1}d_{2}d_{4}\left[  X,\left[  X,\left[  X,Y\right]
\right]  \right]  +\frac{3}{4}d_{1}d_{2}d_{4}\left[  X,\left[  Y,\left[
X,Y\right]  \right]  \right] \nonumber\\
& +\frac{3}{4}d_{1}d_{2}d_{4}\left[  Y,\left[  X,\left[  X,Y\right]  \right]
\right]  +\frac{3}{4}d_{1}d_{2}d_{4}\left[  Y,\left[  Y,\left[  X,Y\right]
\right]  \right]  +\frac{3}{4}d_{1}d_{3}d_{4}\left[  X,\left[  X,\left[
X,Y\right]  \right]  \right] \nonumber\\
& +\frac{3}{4}d_{1}d_{3}d_{4}\left[  X,\left[  Y,\left[  X,Y\right]  \right]
\right]  +\frac{3}{4}d_{1}d_{3}d_{4}\left[  Y,\left[  X,\left[  X,Y\right]
\right]  \right]  +\frac{3}{4}d_{1}d_{3}d_{4}\left[  Y,\left[  Y,\left[
X,Y\right]  \right]  \right] \nonumber\\
& +\frac{3}{4}d_{2}d_{3}d_{4}\left[  X,\left[  X,\left[  X,Y\right]  \right]
\right]  +\frac{3}{4}d_{2}d_{3}d_{4}\left[  X,\left[  Y,\left[  X,Y\right]
\right]  \right]  +\frac{3}{4}d_{2}d_{3}d_{4}\left[  Y,\left[  X,\left[
X,Y\right]  \right]  \right] \nonumber\\
& +\frac{3}{4}d_{2}d_{3}d_{4}\left[  Y,\left[  Y,\left[  X,Y\right]  \right]
\right]  +\frac{1}{2}d_{1}d_{2}d_{3}d_{4}\left[  X,\left[  X,\left[  X,\left[
X,Y\right]  \right]  \right]  \right] \nonumber\\
& +\frac{1}{2}d_{1}d_{2}d_{3}d_{4}\left[  X,\left[  X,\left[  Y,\left[
X,Y\right]  \right]  \right]  \right]  +\frac{1}{2}d_{1}d_{2}d_{3}d_{4}\left[
X,\left[  Y,\left[  X,\left[  X,Y\right]  \right]  \right]  \right]
\nonumber\\
& +\frac{1}{2}d_{1}d_{2}d_{3}d_{4}\left[  X,\left[  Y,\left[  Y,\left[
X,Y\right]  \right]  \right]  \right]  +\frac{1}{2}d_{1}d_{2}d_{3}d_{4}\left[
Y,\left[  X,\left[  X,\left[  X,Y\right]  \right]  \right]  \right]
\nonumber\\
& +\frac{1}{2}d_{1}d_{2}d_{3}d_{4}\left[  Y,\left[  X,\left[  Y,\left[
X,Y\right]  \right]  \right]  \right]  +\frac{1}{2}d_{1}d_{2}d_{3}d_{4}\left[
Y,\left[  Y,\left[  X,\left[  X,Y\right]  \right]  \right]  \right]
\nonumber\\
& +\frac{1}{2}d_{1}d_{2}d_{3}d_{4}\left[  Y,\left[  Y,\left[  Y,\left[
X,Y\right]  \right]  \right]  \right] \label{4.16}%
\end{align}

Letting $m_{54}$ be the right-hand side of (\ref{4.16}), we have
\begin{align}
& \left(  \ref{4.15}\right) \nonumber\\
& =\exp\;-d_{5}m_{51}-d_{5}m_{52}-d_{5}m_{53}-d_{5}m_{54}.\exp\;d_{5}%
m_{54}\nonumber\\
& \exp\;\left(
\begin{array}
[c]{c}%
\left(  d_{1}+d_{2}+d_{3}+d_{4}+d_{5}\right)  \left(  X+Y\right)  +\frac{1}%
{2}\left(  d_{1}+d_{2}+d_{3}+d_{4}\right)  ^{2}\left[  X,Y\right] \\
+\frac{1}{12}\left(  d_{1}+d_{2}+d_{3}+d_{4}\right)  ^{3}\left[  \left[
X,Y\right]  ,Y-X\right] \\
+\frac{1}{96}\left(  d_{1}+d_{2}+d_{3}+d_{4}\right)  ^{4}\left(  \left[
X+Y,\left[  \left[  X,Y\right]  ,X+Y\right]  \right]  +\left[  \left[  \left[
X,Y\right]  ,X-Y\right]  ,X-Y\right]  \right) \\
+d_{5}i_{51}+d_{5}i_{52}+d_{5}i_{53}+d_{5}j_{51}+d_{5}j_{52}%
\end{array}
\right)  .\nonumber\\
& \exp\;-d_{5}n_{51}-d_{5}n_{52}-d_{5}n_{53}-d_{5}n_{54}\nonumber\\
& \left)  \text{By Proposition \ref{t2.2}}\right( \nonumber\\
& =\exp\;-d_{5}m_{51}-d_{5}m_{52}-d_{5}m_{53}-d_{5}m_{54}.\left(
m_{54}\right)  _{d_{5}}\nonumber\\
& \exp\;\left(
\begin{array}
[c]{c}%
\left(  d_{1}+d_{2}+d_{3}+d_{4}+d_{5}\right)  \left(  X+Y\right)  +\frac{1}%
{2}\left(  d_{1}+d_{2}+d_{3}+d_{4}\right)  ^{2}\left[  X,Y\right] \\
+\frac{1}{12}\left(  d_{1}+d_{2}+d_{3}+d_{4}\right)  ^{3}\left[  \left[
X,Y\right]  ,Y-X\right] \\
+\frac{1}{96}\left(  d_{1}+d_{2}+d_{3}+d_{4}\right)  ^{4}\left(  \left[
X+Y,\left[  \left[  X,Y\right]  ,X+Y\right]  \right]  +\left[  \left[  \left[
X,Y\right]  ,X-Y\right]  ,X-Y\right]  \right) \\
+d_{5}i_{51}+d_{5}i_{52}+d_{5}i_{53}+d_{5}j_{51}+d_{5}j_{52}%
\end{array}
\right)  .\nonumber\\
& \exp\;-d_{5}n_{51}-d_{5}n_{52}-d_{5}n_{53}-d_{5}n_{54}\nonumber\\
& \left)  \text{By Proposition \ref{t2.3}}\right( \nonumber\\
& =\exp\;-d_{5}m_{51}-d_{5}m_{52}-d_{5}m_{53}-d_{5}m_{54}.\nonumber\\
& \exp\;\left(
\begin{array}
[c]{c}%
\left(  d_{1}+d_{2}+d_{3}+d_{4}+d_{5}\right)  \left(  X+Y\right)  +\frac{1}%
{2}\left(  d_{1}+d_{2}+d_{3}+d_{4}\right)  ^{2}\left[  X,Y\right] \\
+\frac{1}{12}\left(  d_{1}+d_{2}+d_{3}+d_{4}\right)  ^{3}\left[  \left[
X,Y\right]  ,Y-X\right] \\
+\frac{1}{96}\left(  d_{1}+d_{2}+d_{3}+d_{4}\right)  ^{4}\left(  \left[
X+Y,\left[  \left[  X,Y\right]  ,X+Y\right]  \right]  +\left[  \left[  \left[
X,Y\right]  ,X-Y\right]  ,X-Y\right]  \right) \\
+d_{5}i_{51}+d_{5}i_{52}+d_{5}i_{53}+d_{5}j_{51}+d_{5}j_{52}+d_{5}j_{53}%
\end{array}
\right)  .\nonumber\\
& \exp\;-d_{5}n_{51}-d_{5}n_{52}-d_{5}n_{53}-d_{5}n_{54}\nonumber\\
& \left)  \text{By (\ref{4.16})}\right( \label{4.17}%
\end{align}

Since the coefficient of every term in $-m_{51}-m_{52}-m_{53}-m_{54}$ contains
$d_{1}d_{2}d_{3}d_{4}$, we are done. We have
\begin{align*}
& i_{51}+i_{52}+i_{53}+j_{51}+j_{52}+j_{53}\\
& =d_{1}\left[  X,Y\right]  +d_{2}\left[  X,Y\right]  -\frac{1}{2}d_{1}%
d_{2}\left[  \left[  X,Y\right]  ,X\right]  +\frac{1}{2}d_{1}d_{2}\left[
\left[  X,Y\right]  ,Y\right]  +d_{3}\left[  X,Y\right] \\
& -\frac{1}{2}d_{1}d_{3}\left[  \left[  X,Y\right]  ,X\right]  +\frac{1}%
{2}d_{1}d_{3}\left[  \left[  X,Y\right]  ,Y\right]  -\frac{1}{2}d_{2}%
d_{3}\left[  \left[  X,Y\right]  ,X\right] \\
& +\frac{1}{2}d_{2}d_{3}\left[  \left[  X,Y\right]  ,Y\right]  +\frac{1}%
{4}d_{1}d_{2}d_{3}\left[  X,\left[  \left[  X,Y\right]  ,X\right]  \right]
+\frac{1}{4}d_{1}d_{2}d_{3}\left[  X,\left[  \left[  X,Y\right]  ,Y\right]
\right] \\
& +\frac{1}{4}d_{1}d_{2}d_{3}\left[  Y,\left[  \left[  X,Y\right]  ,X\right]
\right]  +\frac{1}{4}d_{1}d_{2}d_{3}\left[  Y,\left[  \left[  X,Y\right]
,Y\right]  \right]  +\frac{1}{4}d_{1}d_{2}d_{3}\left[  \left[  \left[
X,Y\right]  ,X\right]  ,X\right] \\
& -\frac{1}{4}d_{1}d_{2}d_{3}\left[  \left[  \left[  X,Y\right]  ,X\right]
,Y\right]  -\frac{1}{4}d_{1}d_{2}d_{3}\left[  \left[  \left[  X,Y\right]
,Y\right]  ,X\right]  +\frac{1}{4}d_{1}d_{2}d_{3}\left[  \left[  \left[
X,Y\right]  ,Y\right]  ,Y\right] \\
& +d_{4}\left[  X,Y\right]  -\frac{1}{2}d_{1}d_{4}\left[  \left[  X,Y\right]
,X\right]  +\frac{1}{2}d_{1}d_{4}\left[  \left[  X,Y\right]  ,Y\right]
-\frac{1}{2}d_{2}d_{4}\left[  \left[  X,Y\right]  ,X\right] \\
& +\frac{1}{2}d_{2}d_{4}\left[  \left[  X,Y\right]  ,Y\right]  +\frac{1}%
{4}d_{1}d_{2}d_{4}\left[  X,\left[  \left[  X,Y\right]  ,X\right]  \right]
+\frac{1}{4}d_{1}d_{2}d_{4}\left[  X,\left[  \left[  X,Y\right]  ,Y\right]
\right] \\
& +\frac{1}{4}d_{1}d_{2}d_{4}\left[  Y,\left[  \left[  X,Y\right]  ,X\right]
\right]  +\frac{1}{4}d_{1}d_{2}d_{4}\left[  Y,\left[  \left[  X,Y\right]
,Y\right]  \right]  +\frac{1}{4}d_{1}d_{2}d_{4}\left[  \left[  \left[
X,Y\right]  ,X\right]  ,X\right] \\
& -\frac{1}{4}d_{1}d_{2}d_{4}\left[  \left[  \left[  X,Y\right]  ,X\right]
,Y\right]  -\frac{1}{4}d_{1}d_{2}d_{4}\left[  \left[  \left[  X,Y\right]
,Y\right]  ,X\right]  +\frac{1}{4}d_{1}d_{2}d_{4}\left[  \left[  \left[
X,Y\right]  ,Y\right]  ,Y\right] \\
& -\frac{1}{2}d_{3}d_{4}\left[  \left[  X,Y\right]  ,X\right]  +\frac{1}%
{2}d_{3}d_{4}\left[  \left[  X,Y\right]  ,Y\right]  +\frac{1}{4}d_{1}%
d_{3}d_{4}\left[  X,\left[  \left[  X,Y\right]  ,X\right]  \right] \\
& +\frac{1}{4}d_{1}d_{3}d_{4}\left[  X,\left[  \left[  X,Y\right]  ,Y\right]
\right]  +\frac{1}{4}d_{1}d_{3}d_{4}\left[  Y,\left[  \left[  X,Y\right]
,X\right]  \right]  +\frac{1}{4}d_{1}d_{3}d_{4}\left[  Y,\left[  \left[
X,Y\right]  ,Y\right]  \right] \\
& +\frac{1}{4}d_{1}d_{3}d_{4}\left[  \left[  \left[  X,Y\right]  ,X\right]
,X\right]  -\frac{1}{4}d_{1}d_{3}d_{4}\left[  \left[  \left[  X,Y\right]
,X\right]  ,Y\right]  -\frac{1}{4}d_{1}d_{3}d_{4}\left[  \left[  \left[
X,Y\right]  ,Y\right]  ,X\right] \\
& +\frac{1}{4}d_{1}d_{3}d_{4}\left[  \left[  \left[  X,Y\right]  ,Y\right]
,Y\right]  +\frac{1}{4}d_{2}d_{3}d_{4}\left[  X,\left[  \left[  X,Y\right]
,X\right]  \right]  +\frac{1}{4}d_{2}d_{3}d_{4}\left[  X,\left[  \left[
X,Y\right]  ,Y\right]  \right] \\
& +\frac{1}{4}d_{2}d_{3}d_{4}\left[  Y,\left[  \left[  X,Y\right]  ,X\right]
\right]  +\frac{1}{4}d_{2}d_{3}d_{4}\left[  Y,\left[  \left[  X,Y\right]
,Y\right]  \right]  +\frac{1}{4}d_{2}d_{3}d_{4}\left[  \left[  \left[
X,Y\right]  ,X\right]  ,X\right] \\
& -\frac{1}{4}d_{2}d_{3}d_{4}\left[  \left[  \left[  X,Y\right]  ,X\right]
,Y\right]  -\frac{1}{4}d_{2}d_{3}d_{4}\left[  \left[  \left[  X,Y\right]
,Y\right]  ,X\right]  +\frac{1}{4}d_{2}d_{3}d_{4}\left[  \left[  \left[
X,Y\right]  ,Y\right]  ,Y\right]
\end{align*}
on the one hand, and
\begin{align*}
& -m_{51}-m_{52}-m_{53}-m_{54}-n_{51}-n_{52}-n_{53}-n_{54}\\
& =d_{1}d_{2}d_{3}d_{4}\left[  X,\left[  X,\left[  \left[  X,Y\right]
,X\right]  \right]  \right]  -d_{1}d_{2}d_{3}d_{4}\left[  X,\left[  X,\left[
\left[  X,Y\right]  ,Y\right]  \right]  \right] \\
& +d_{1}d_{2}d_{3}d_{4}\left[  X,\left[  Y,\left[  \left[  X,Y\right]
,X\right]  \right]  \right]  -d_{1}d_{2}d_{3}d_{4}\left[  X,\left[  Y,\left[
\left[  X,Y\right]  ,Y\right]  \right]  \right] \\
& +d_{1}d_{2}d_{3}d_{4}\left[  Y,\left[  X,\left[  \left[  X,Y\right]
,X\right]  \right]  \right]  -d_{1}d_{2}d_{3}d_{4}\left[  Y,\left[  X,\left[
\left[  X,Y\right]  ,Y\right]  \right]  \right] \\
& +d_{1}d_{2}d_{3}d_{4}\left[  Y,\left[  Y,\left[  \left[  X,Y\right]
,X\right]  \right]  \right]  -d_{1}d_{2}d_{3}d_{4}\left[  Y,\left[  Y,\left[
\left[  X,Y\right]  ,Y\right]  \right]  \right] \\
& -\frac{1}{6}d_{1}d_{2}d_{3}d_{4}\left[  X,\left[  \left[  \left[
X,Y\right]  ,X\right]  ,X\right]  \right]  -\frac{1}{6}d_{1}d_{2}d_{3}%
d_{4}\left[  X,\left[  \left[  \left[  X,Y\right]  ,X\right]  ,Y\right]
\right] \\
& +\frac{1}{6}d_{1}d_{2}d_{3}d_{4}\left[  X,\left[  \left[  \left[
X,Y\right]  ,Y\right]  ,X\right]  \right]  +\frac{1}{6}d_{1}d_{2}d_{3}%
d_{4}\left[  X,\left[  \left[  \left[  X,Y\right]  ,Y\right]  ,Y\right]
\right] \\
& -\frac{1}{6}d_{1}d_{2}d_{3}d_{4}\left[  Y,\left[  \left[  \left[
X,Y\right]  ,X\right]  ,X\right]  \right]  -\frac{1}{6}d_{1}d_{2}d_{3}%
d_{4}\left[  Y,\left[  \left[  \left[  X,Y\right]  ,X\right]  ,Y\right]
\right] \\
& +\frac{1}{6}d_{1}d_{2}d_{3}d_{4}\left[  Y,\left[  \left[  \left[
X,Y\right]  ,Y\right]  ,X\right]  \right]  +\frac{1}{6}d_{1}d_{2}d_{3}%
d_{4}\left[  Y,\left[  \left[  \left[  X,Y\right]  ,Y\right]  ,Y\right]
\right] \\
& +\frac{1}{2}d_{1}d_{2}d_{3}d_{4}\left[  \left[  X,Y\right]  ,\left[  \left[
X,Y\right]  ,X\right]  \right]  +\frac{1}{2}d_{1}d_{2}d_{3}d_{4}\left[
\left[  X,Y\right]  ,\left[  \left[  X,Y\right]  ,Y\right]  \right] \\
& -\frac{1}{8}d_{1}d_{2}d_{3}d_{4}\left[  \left[  X,\left[  \left[
X,Y\right]  ,X\right]  \right]  ,X\right]  +\frac{1}{8}d_{1}d_{2}d_{3}%
d_{4}\left[  \left[  X,\left[  \left[  X,Y\right]  ,X\right]  \right]
,Y\right] \\
& -\frac{1}{8}d_{1}d_{2}d_{3}d_{4}\left[  \left[  X,\left[  \left[
X,Y\right]  ,Y\right]  \right]  ,X\right]  +\frac{1}{8}d_{1}d_{2}d_{3}%
d_{4}\left[  \left[  X,\left[  \left[  X,Y\right]  ,Y\right]  \right]
,Y\right] \\
& -\frac{1}{8}d_{1}d_{2}d_{3}d_{4}\left[  \left[  Y,\left[  \left[
X,Y\right]  ,X\right]  \right]  ,X\right]  +\frac{1}{8}d_{1}d_{2}d_{3}%
d_{4}\left[  \left[  Y,\left[  \left[  X,Y\right]  ,X\right]  \right]
,Y\right] \\
& -\frac{1}{8}d_{1}d_{2}d_{3}d_{4}\left[  \left[  Y,\left[  \left[
X,Y\right]  ,Y\right]  \right]  ,X\right]  +\frac{1}{8}d_{1}d_{2}d_{3}%
d_{4}\left[  \left[  Y,\left[  \left[  X,Y\right]  ,Y\right]  \right]
,Y\right] \\
& -\frac{1}{8}d_{1}d_{2}d_{3}d_{4}\left[  \left[  \left[  \left[  X,Y\right]
,X\right]  ,X\right]  ,X\right]  +\frac{1}{8}d_{1}d_{2}d_{3}d_{4}\left[
\left[  \left[  \left[  X,Y\right]  ,X\right]  ,X\right]  ,Y\right] \\
& +\frac{1}{8}d_{1}d_{2}d_{3}d_{4}\left[  \left[  \left[  \left[  X,Y\right]
,X\right]  ,Y\right]  ,X\right]  -\frac{1}{8}d_{1}d_{2}d_{3}d_{4}\left[
\left[  \left[  \left[  X,Y\right]  ,X\right]  ,Y\right]  ,Y\right] \\
& +\frac{1}{8}d_{1}d_{2}d_{3}d_{4}\left[  \left[  \left[  \left[  X,Y\right]
,Y\right]  ,X\right]  ,X\right]  -\frac{1}{8}d_{1}d_{2}d_{3}d_{4}\left[
\left[  \left[  \left[  X,Y\right]  ,Y\right]  ,X\right]  ,Y\right] \\
& -\frac{1}{8}d_{1}d_{2}d_{3}d_{4}\left[  \left[  \left[  \left[  X,Y\right]
,Y\right]  ,Y\right]  ,X\right]  +\frac{1}{8}d_{1}d_{2}d_{3}d_{4}\left[
\left[  \left[  \left[  X,Y\right]  ,Y\right]  ,Y\right]  ,Y\right]
\end{align*}
on the other. Therefore we have the desired result.
\end{proof}

\end{document}